\documentclass[a4paper,12pt]{article}

\usepackage{amsmath,amssymb,graphicx,amsthm,ifthen,epstopdf,fancyhdr,color}
\usepackage[margin=3cm,vmargin={2cm,3.5cm},includefoot]{geometry}
\usepackage[bookmarks=true]{hyperref}
\usepackage{palatino}
\usepackage{cite}
\usepackage{fancyhdr}
\theoremstyle{theorem} \newtheorem{thm}{Theorem}
\theoremstyle{theorem} \newtheorem{cor}{Corollary}
\theoremstyle{theorem} \newtheorem{defn}{Definition}
\theoremstyle{theorem} \newtheorem{lemma}{Lemma}
\theoremstyle{theorem} 
\newtheorem*{theorem*}{Theorem}
\newtheorem*{corollary*}{Corollary}
\usepackage{color}
\usepackage{ifthen}

  \newcommand{\R}{\ensuremath{\mathbb{R}}}
  \newcommand{\E}{\ensuremath{\mathbb{E}}}
  
\newcommand{\V}[1]{\ensuremath{\mathbf{#1}}}

\newcommand{\norm}[1]{\left|\left| #1 \right|\right|}

\newcommand{\NOTE}[1]{ 
\ifx\NOTES\undefined\else
  \footnote{ {\color{blue} NOTE: #1}}  
\fi
}

\newcommand{\aslim}{\stackrel{a.s.}{\rightarrow}}

\newcommand{\Pd}[3]{\ifthenelse{\equal{#3}{1}}{\frac{\partial #1}{\partial #2}}{\frac{\partial^{#3} #1}{\partial #2^{#3}}}}

\newcommand{\n}{n} 
\newcommand{\m}{m} 
\newcommand{\B}{\beta} 
\newcommand{\rk}{r} 
\newcommand{\Mnn}{M_{\n\times\n}}
\newcommand{\Mmn}{M_{\m\times\n}}
\newcommand{\Mmm}{M_{\m\times\m}}

\newcommand{\inner}[2]{\ensuremath{\langle #1\,,\,#2 \rangle}}
\newcommand{\aseq}{\stackrel{a.s.}{=}}

\newcommand{\ct}{\tilde{c}}
\newcommand{\stt}{\tilde{s}}
\newcommand{\ut}{\tilde{u}}
\newcommand{\vt}{\tilde{v}}
\newcommand{\wt}{\tilde{w}}
\newcommand{\Wt}{\tilde{W}}
\newcommand{\Ut}{\tilde{U}}
\newcommand{\Vt}{\tilde{V}}

\newcommand{\thename}{Optimal Shrinkage of Singular Values}

\begin{document}

\title{\thename}

\author{
    Matan Gavish \footnotemark[2]
    \and 
    David L. Donoho \footnotemark[1]
}

\date{}

\maketitle

\renewcommand{\thefootnote}{\fnsymbol{footnote}}
\footnotetext[1]{Department of Statistics, Stanford University}
\footnotetext[2]{School of Computer Science and Engineering, Hebrew University of
Jerusalem}
\renewcommand{\thefootnote}{\arabic{footnote}}

\begin{abstract} 
  We consider recovery of low-rank matrices from noisy data by shrinkage of
  singular values, in which a single, univariate nonlinearity is applied to each
  of the empirical singular values. We adopt an asymptotic framework, in which
  the matrix size is much larger than the rank of the signal matrix to be
  recovered, and the signal-to-noise ratio of the low-rank piece stays constant.
  For a variety of loss functions, including Mean Square Error (MSE -- square
  Frobenius norm), the nuclear norm loss and the operator norm loss, we show
  that in this framework there is a well-defined asymptotic loss that we
  evaluate precisely in each case.  In fact, each of the loss functions we study
  admits a {\em unique admissible} shrinkage nonlinearity dominating all other
  nonlinearities.  We provide a general method for evaluating these optimal
  nonlinearities, and demonstrate our framework by working out simple, explicit
  formulas for the optimal nonlinearities in the Frobenius, nuclear and operator
  norm cases.  For example, for a square low-rank $n$-by-$n$ matrix observed in
  white noise with level $\sigma$, the optimal nonlinearity for MSE loss simply
  shrinks each data singular value $y$ to $\sqrt{y^2-4n\sigma^2 }$ (or to $0$ if
  $y<2\sqrt{n}\sigma$).  This optimal nonlinearity guarantees an asymptotic MSE
  of $2nr\sigma^2$, which compares favorably with optimally tuned hard
  thresholding and optimally tuned soft thresholding, providing guarantees of
  $3nr\sigma^2$ and $6nr\sigma^2$, respectively. Our general method also allows
  one to evaluate optimal shrinkers numerically to arbitrary precision.  As an
  example, we compute optimal shrinkers for the Schatten-$p$ norm loss, for any
  $p>0$.
\end{abstract}

\maketitle
\newpage
\tableofcontents
\newpage
\section{Introduction}

Suppose that we are interested in an $m$-by-$n$ matrix $X$, which is thought to be
either exactly or approximately of low rank, but we only observe a single noisy
$m$-by-$n$ matrix $Y$, obeying $Y=X+\sigma Z$; The noise matrix $Z$  has
independent, identically distributed entries with zero mean, unit variance,
and a finite fourth moment. 
We choose a loss function $L_{m,n}(\cdot,\cdot)$ and wish to recover the matrix
$X$ with some bound on the risk $\E L_{m,n}(X,\hat{X})$, where $\hat{X}$ is our
estimated value of $X$. 

For example, when choosing the square Frobenius loss, or mean square error (MSE) 
\begin{eqnarray} \label{L_fro:eq}
  L_{m,n}^{fro}(X,\hat{X}) =
  \norm{X-\hat{X}}_F^2=\sum_{i,j}|X_{i,j}-\hat{X}_{i,j}|^2\,,
\end{eqnarray}
where $X$ and $\hat{X}$  are $m$-by-$n$ matrices,
we would like to find an
estimator $\hat{X}$ with small mean square error (MSE). 
The default technique for estimating a low rank matrix in noise is the {\em Truncated SVD} (TSVD)
\cite{Golub1965}: write
\begin{eqnarray} \label{svd:eq}
  Y = \sum_{i=1}^m y_i \V{v}_{i} \V{\vt}_{i}'
\end{eqnarray}
for the Singular Value Decomposition of the data matrix $Y$, where
$\V{v}_i\in\R^m$ and $\V{\vt}_i\in\R^n$ (for $i=1,\ldots,m$) are the left and right
singular vectors
of $Y$ corresponding to the singular value $y_i$.
The TSVD estimator is 
\begin{eqnarray*}
  \hat{X}_\rk = \sum_{i=1}^\rk y_i \V{v}_i \V{\vt}_i'\,,
\end{eqnarray*}
where $\rk=rank(X)$, assumed known, and $y_1\geq \ldots \geq y_m$.  Being the
best approximation of rank $\rk$ to the data in the least squares sense
\cite{Eckart1936}, and therefore the Maximum Likelihood estimator when $Z$ has
Gaussian entries, the TSVD is arguably as ubiquitous in science and engineering
as linear regression
\cite{Alter2000,Cattell1966,Jackson1993,Lagerlund1997,Price2006,Edfors1998}.

The TSVD estimator shrinks to zero some of the data singular values, while
leaving others untouched.  More generally, for any specific choice of scalar
nonlinearity $\eta:[0,\infty)\to [0,\infty)$, also known as a shrinker, there is
a corresponding {\em singular value shrinkage} estimator $\hat{X}_\eta$ given
by  
\begin{eqnarray} \label{shrink:eq} 
  \hat{X}_\eta = \sum_{i=1}^m \eta(y_i)
    \V{v}_{i} \V{\vt}_{i}'\,.  
\end{eqnarray}

For scalar and vector denoising, univariate shrinkage rules  have proved to be
simple and practical denoising methods, with near-optimal performance guarantees
under various performance measures
\cite{Donoho1994b,Donoho1995,Donoho1995b,Donoho1995c,Donoho1998}.  
Shrinkage makes sense for singular values, too: presumably, the observed
singular values $y_1 \ldots y_m$ are ``inflated'' by the noise, and applying a
carefully chosen shrinkage function, one can obtain a good estimate of the
original signal $X$. 

Singular value shrinkage arises when the estimator $\hat{X}$ for the signal $X$
has to be bi-orthogonally invariant under rotations of the data matrix. The most
general form of an invariant estimator is $\hat{X}=V \hat{D} \tilde{V}'$, where
the matrices $V$ and $\tilde{V}$ contain the left and right singular vectors of
the data, and where $\hat{D}$ is a diagonal matrix that depends on the data
singular values. In other words, the most general invariant estimator is
equivalent to a vector map $\R^m\to\R^m$ acting on the vector of data singular
values. This is a wide and complicated class of estimators; focusing on
univariate singular value shrinkage \eqref{shrink:eq} allows for a simpler
discussion\footnote{However, it is interesting to remark that, at least in the
  Frobenius loss case, and possibly in other cases as well, the asymptotically
  optimal univariate shrinker, presented in this paper, offers the same
  performance asymptotically as the  best possible invariant estimator of the
  form $\hat{X}=V\hat{D}\tilde{V}'$ - see \cite{Donohoa}.}.

Indeed, there is a growing body of literature on matrix denoising by shrinkage
of singular values, going back, to the best of our knowledge, to Owen and Perry
\cite{Owen2009,Perry2009} and Shabalin and Nobel \cite{Shabalin2013}.  Soft
thresholding of singular values has been considered in
\cite{Cai2008,DonohoGavish2013,Donoho2013b}, and hard thresholding in
\cite{Chatterjee2010,Donoho2013b}. In fact, \cite{Perry2009,Shabalin2013} and,
very recently, \cite{Nadakuditi,Josse2016} considered shrinkers that are
developed specifically for singular values, and measured their performance using
Frobenius loss.

These developments suggest the following question:
{\em Is there a simple, natural shrinkage nonlinearity for singular values?} If
there is a simple answer to this question, surely it depends on the loss function $L$ and on
specific assumptions on the signal matrix $X$. 

In \cite{Donoho2013b} we have performed a narrow investigation that focused on
hard and soft thresholding of singular values under the  Frobenius loss
\eqref{L_fro:eq}.  We adopted a simple asymptotic framework that models the
situation where $X$ is low-rank, originally proposed in
\cite{Owen2009,Perry2009,Shabalin2013} and inspired by Johnstone's Spiked
Covariance Model \cite{Johnstone2001}. In this framework, the signal matrix
dimensions $m=m_n$ and $n$ both to infinity, such that their ratio converges to
an asymptotic aspect ratio: $m_n/n\to\beta$, with $0<\beta\leq 1$, while the
column span of the signal matrix remains fixed.  Building on a recent
probabilistic analysis of this framework \cite{Benaych-Georges2012} we have
discovered that, in this framework, there is an {\em asymptotically unique
admissible} threshold for singular values, 
in the sense that it offers equal or
better asymptotic MSE to that of any other threshold choice,
no matter which specific low-rank model
may be in force.

The main discovery reported here is that this phenomenon is in fact much more
general: in this asymptotic framework, which models low-rank matrices observed
in white noise, for each of a  variety of loss functions, there exists a single
{\em asymptotically unique admissible} shrinkage nonlinearity, 
in the sense that
it offers equal or better asymptotic loss than any other shrinkage
nonlinearity, at each specific low-rank model
that can occur.
In other words, once
the loss function has been decided, in a definite asymptotic sense, there is a
single rational choice of shrinkage nonlinearity. 

\subsection*{Some optimal shrinkers} \label{some:subsec}
In this paper, we
develop a general method for finding the optimal shrinkage nonlinearity for a
variety of loss functions. We explicitly work out the optimal shrinkage
formula for the Frobenius norm loss, the nuclear norm loss, and the operator
norm loss. Let us denote the Frobenius, Operator and Nuclear matrix norms by 
$\norm{\cdot}_{F}$,$\norm{\cdot}_{op}$ and $\norm{\cdot}_{*}$, respectively.
If the singular values of the matrix $X-\hat{X}$ are 
$\sigma_1,\ldots,\sigma_m$, then these losses are given by
\begin{eqnarray}
  L_{m,n}^{fro}(X,\hat{X}) &=&   \norm{X-\hat{X}}_F^2 = \sum_{i=1}^m \sigma_i^2 \label{eq:fro}\\
  L_{m,n}^{op}(X,\hat{X}) &=&  \norm{X-\hat{X}}_{op} = \max\left\{ \sigma_1,\ldots,\sigma_m \right\} \label{eq:op}\\
  L_{m,n}^{nuc}(X,\hat{X}) &=& \norm{X-\hat{X}}_* = \sum_{i=1}^m \sigma_i  \label{eq:nuc}
  \,.
\end{eqnarray}

\paragraph{Optimal shrinker for the Frobenius norm loss.} 

As we will see, the optimal nonlinearity for the Frobenius norm loss
\eqref{eq:fro}, in a natural noise scaling, is 
\begin{eqnarray} \label{opt_shrinker_fro:eq}
  \eta^*(y) = 
    \begin{cases}
      \tfrac{1}{y}\sqrt{(y^2-\B-1)^2-4\B} & y \geq 1+\sqrt{\beta} \\
      0 & y \leq 1+\sqrt{\beta}
    \end{cases}\,\,.
\end{eqnarray}
In the asymptotically square case $\B=1$ this reduces to 
\[
    \eta(y) = \sqrt{(y^2-4)_+}\,.
\]

\paragraph{Optimal shrinker for Operator norm loss.}

The operator norm loss \eqref{eq:op} for matrix estimation 
has mostly been studied in the context of covariance estimation
\cite{Bickel2008a,Lam2009,Cai2010a}.
 Let us define
 \begin{eqnarray} \label{x_of_y:eq}
   x(y) = \begin{cases}
     \tfrac{1}{\sqrt{2}}\sqrt{y^2-\B-1 + \sqrt{(y^2-\B-1)^2-4\B} } & y\geq
     1+\sqrt{\beta} \\
    0 & y \leq 1+\sqrt{\beta}
   \end{cases}\,\,.
 \end{eqnarray}

 As we will see, the optimal nonlinearity for operator loss is just
 \begin{eqnarray} \label{opt_shrinker_op:eq}
   \eta^*(y) = x(y)\,.
   \end{eqnarray}

\paragraph{Optimal Shrinkage for Nuclear norm loss.} 

The Nuclear norm loss \eqref{eq:nuc} has also been proposed for matrix
estimation.  
See \cite{Rohde2011,Koltchinskii2011a} and references within for discussion of the Nuclear norm
 and, more generally, of Schatten-$p$ norms as losses for matrix estimation.

 As we will see, the optimal nonlinearity for nuclear norm loss is 

 \begin{eqnarray} \label{opt_shrinker_nuc:eq}
\eta^*(y) = 
    \begin{cases}
      \frac{1}{x^2y}(x^4 - \beta - \sqrt{\beta}xy) & \,\,\,x^4 \geq \beta +\sqrt{\beta}xy 
 \\
 0 &\,\,\,  x^4 < \beta + \sqrt{\beta}xy 
    \end{cases}\,\,.
 \end{eqnarray}
 where $x=x(y)$ is given in \eqref{x_of_y:eq}.
 \\~\\
Note that the formulas above are
 calibrated for the natural noise level $\sigma=1/\sqrt{n}$; see Section \ref{noise:sec} below 
 for usage in known noise level $\sigma$ or unknown noise level. In the code
 supplement for this paper  \cite{Gavish2015} we offer a Matlab implementation 
 of each of these shrinkers in known or unknown noise.

Figure \ref{shrinkers:fig} shows  
the three nonlinearities
\eqref{opt_shrinker_fro:eq}, \eqref{opt_shrinker_op:eq} and
\eqref{opt_shrinker_nuc:eq}. As we will see, these nonlinearities, and many
others that are not calculated explicitly in this paper, 
flow from a single general method for calculating
optimal nonlinearities, developed here.

 \begin{figure}[h!]
   \centerline{
     \includegraphics[height=5.8in]{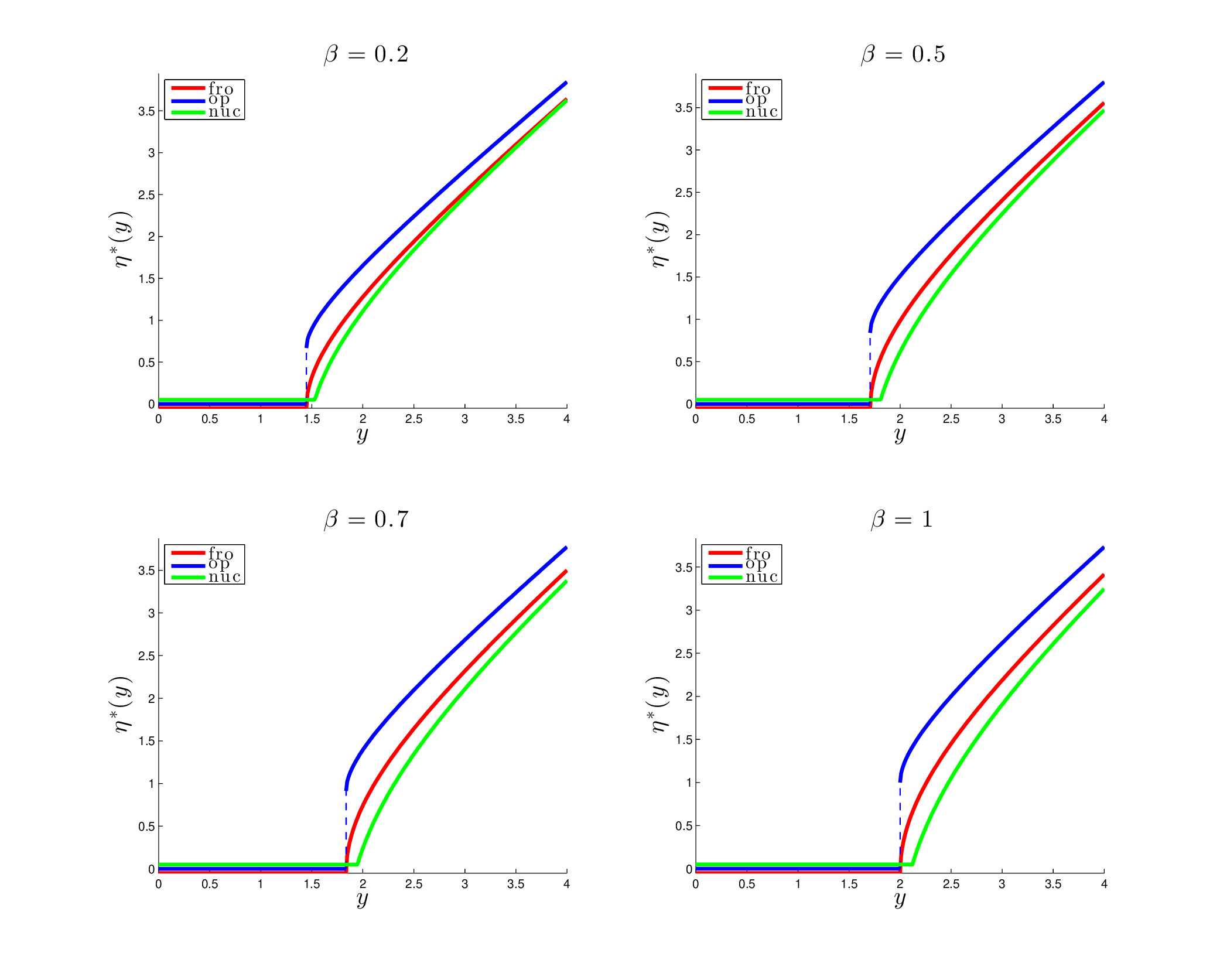}}
   \caption{\small Optimal shrinkers for Frobenius, Operator and Nuclear norm losses
   for different values of $\beta$. All shrinkers 
   asymptote to the identity $\eta(y)=y$ as $y\to\infty$. Curves
 jittered in the vertical axis to avoid overlap. This figure can be reproduced
 using the code supplement \cite{Gavish2015}.}
   \label{shrinkers:fig}
 \end{figure}

 \subsection{Optimal shrinkers vs. hard and soft thresholding}
 \label{hard_and_soft:subsec}

 The optimal shrinkers presented have simple, closed-form formulas. Yet there
 are shrinkage rules that are simpler still, namely, 
 hard and soft thresholding. These nonlinearities are extremely popular 
 for scalar and vector
 denoising, due 
 to their simplicity and various optimality
 properties
 \cite{Donoho1994b,Donoho1995,Donoho1995b,Donoho1995c,Donoho1998}.  
Recall that for $y\geq 0$,
 \begin{eqnarray*}
   \eta^{soft}_s(y) &=&  \max\left( 0, y-s \right) \\
   \eta^{hard}_\lambda(y) &=& y\cdot \mathbf{1}_{y\geq \lambda}\,.
 \end{eqnarray*}
 
It is worthwhile to ask how our optimal shrinkers differ, in shape and
performance, from the popular hard and soft thresholding.  To make a comparison,
one should first decide how to tune the thresholds $\lambda$ and $s$. In our
asymtotic framework, fortunately, there is a decisive answer to the tuning
question: in previous work \cite{Donoho2013b}, we have restricted our attention
to hard and soft thresholding under the Frobenius loss \eqref{eq:fro}. It was
shown that there exist optimal values $\lambda_*(\beta)$ and $s_*(\beta)$, which
are unique admissible in the sense that they offer asymptotic performance equal
to or better than the performance of any other thresold. The optimal thresholds
are given by
\begin{eqnarray*}
  \lambda_*(\beta) &=&  
 \sqrt{2(\beta+1)+\frac{8\beta}{(\beta+1)+\sqrt{\beta^2+14 \beta+1}}}  \\
 s_*(\beta) &=&  1+\sqrt{\beta}\,,
\end{eqnarray*}
where again $\beta$ is the limiting aspect ratio, $m_n/n\to\beta$.

Consider, for example, the square matrix case $\beta=1$. Under the MSE loss,
the optimal hard threshold is then $\lambda_*=4/\sqrt{3}$, and the optimal soft
threshold is $s_*=2$.  Figure 2 shows the nonlinearities
$\eta^{hard}_{\lambda_*}$ and $\eta^{soft}_{s_*}$ against our optimal shrinkers
\eqref{opt_shrinker_fro:eq}, \eqref{opt_shrinker_op:eq} and
\eqref{opt_shrinker_nuc:eq}. In high SNR ($y\gg 1$) the optimal shrinkers agree
with hard thresholding and neither performs any shrinkage, while soft
thresholding shrinks even strong signals.  
As shown in \cite{Donoho2013b}, the worst-case asymptotic MSE over a rank-$r$
matrix observed in noise level $1/\sqrt{n}$ is $2r$ for our optimal shrinker  
\eqref{opt_shrinker_fro:eq}, $3r$ for the optimally tuned hard thresholding
nonlinearity $\eta^{hard}_{\lambda_*}$ and $6r$ for the optimally tuned soft
thresholding nonlinearity $\eta^{soft}_{s_*}$. Hard thresholding is worse in
 intermediate SNR levels; Soft thresholding is worse in strong SNR. 
 For further discussion on this
 phenomenon, which stems from the random rotation of the data 
 singular vectors due to noise, see \cite{Donoho2013b}. We conclude that optimal
 shrinkage, developed in this paper, offers significant performance
 improvement over hard
 and soft thresholding - even when they are optimally tuned.

\begin{figure}[h!]
\centerline{
  \includegraphics[height=1.9in]{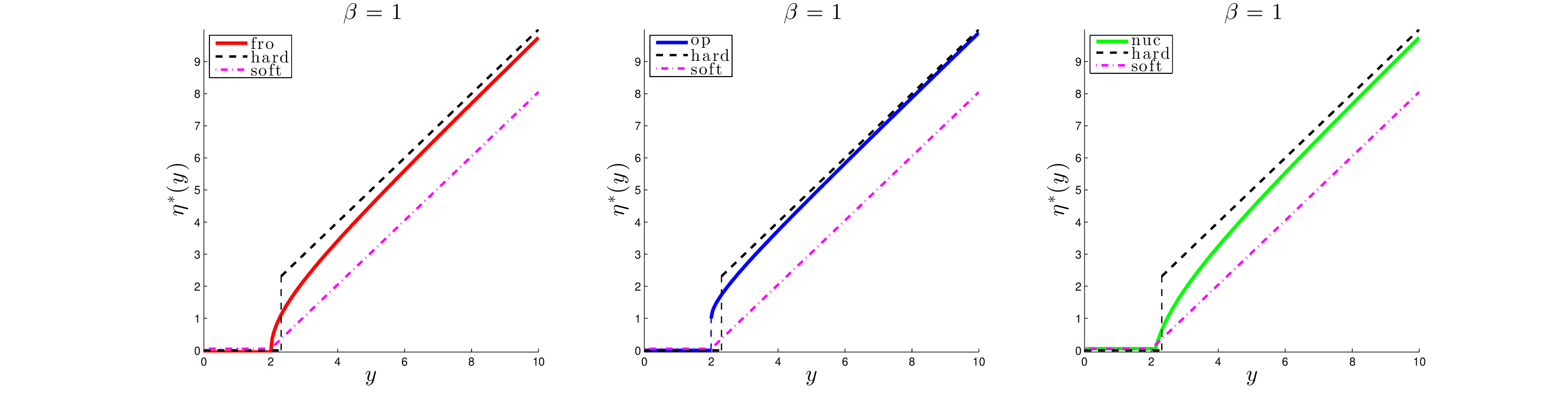}}
   \caption{\small 
     Optimal shrinkers for Frobenius (left), Operator (center)
     and Nuclear norm (right) losses for square matrices ($\beta=1$) 
     plotted with optimally tuned 
     hard and soft thresholding nonlinearities.
     This figure can be reproduced
 using the code supplement \cite{Gavish2015}.}

   \label{shrinkers:fig}
 \end{figure}

\section{Preliminaries}

Column vectors are denoted by boldface lowercase letters, such as $\V{v}$, their
transpose is $\V{v}'$ and their $i$-th coordinate is $v_i$. The Euclidean inner
product and norm on vectors are denoted by $\inner{\V{u}}{\V{v}}$ and
$\norm{\V{u}}_2$, respectively.  Matrices are denotes by uppercase letters, such
as $X$, their transpose is $X'$ and their $i,j$-th entry is $X_{i,j}$.  $\Mmn$
denotes the space of real $m$-by-$n$ matrices,
$\inner{X}{Y}=\sum_{i,j}X_{i,j}Y_{i,j}$ denotes the Hilbert–-Schmidt inner
product and $\norm{X}_F$ denotes the corresponding Frobenius norm on $\Mmn$.
$\norm{X}_*$ and $\norm{X}_{op}$ denote the nuclear norm (sum of singular
values) and operator norm (maximal singular value) of $X$, respectively.  For
simplicity we only consider $m\leq n$.  We denote matrix denoisers, or
estimators, by $\hat{X}:\Mmn\to\Mmn$.  The symbols $\aslim$ and $\aseq$ denote
almost sure convergence and equality of a.s. limits, respectively.
We use ``fat'' SVD of a matrix $X\in\Mmn$ with $m\leq n$, that is, when writing
$X=UD\Ut$ we mean that $U\in\Mmm$, $D\in\Mmn$, and $\Ut\in\Mnn$. 
Symbols without tilde such as $\V{u}$ are associated with left singular
vectors, while symbols with  tilde such as $\V{\ut}$ are associated with right
singular vectors. By $diag(x_1,\ldots,x_m)$ we mean the $m$-by-$n$ matrix whose
main diagonal is $x_1,\ldots,x_m$, with $n$ implicit in the notation and
inferred from context.

\subsection{Natural problem scaling}

In the general model $Y=X+\sigma Z$, the noise level in the singular values of
$Y$ is $\sqrt{n}\sigma$.  Instead of specifying a different shrinkage rule that
depends on the matrix size $n$, we calibrate our shrinkage rules to the
``natural'' model $Y=X+Z/\sqrt{n}$. In this convention, shrinkage rules stay the
same for every value of $n$, and we conveniently abuse notation by writing
$\hat{X}_\eta$ as in \eqref{shrink:eq} for any $\hat{X}_\eta:\Mmn\to\Mmn$,
keeping $m$ and $n$ implicit.  To apply any denoiser $\hat{X}$ below to data
from the general model $Y=X+\sigma Z$, use the denoiser 
\begin{eqnarray} \label{scale-xhat:eq} 
  \hat{X}_\eta^{(n,\sigma)}(Y) = \sqrt{n}\sigma \cdot
  \hat{X}_\eta(Y/\sqrt{n}\sigma)\,.  
\end{eqnarray}

Throughout the text, we use $\hat{X}_\eta$ to denote singular value shrinker calibrated for
noise level $1/\sqrt{n}$. In Section \ref{noise:sec} below we provide a
recipe for applying any denoiser $\hat{X}_\eta$ calibrated for noise level
$\sigma=1/\sqrt{n}$ for data in the presence of unknown noise level.

\subsection{Asymptotic framework and problem statement} \label{framework:subsec}

In this paper, we consider a sequence of increasingly larger denoising problems 
\begin{eqnarray} \label{model:eq}
  Y_n=X_n+Z_n/\sqrt{n}\, 
\end{eqnarray}
with 
$X_n,Z_n\in M_{m_n,n}$, satisfying the
following assumptions:
\begin{enumerate}

\item {\em Invariant white noise:} The entries of $Z_n$ are i.i.d samples from a
distribution with zero mean, unit variance and finite fourth moment.  To
simplify the formal statement of our results, we assume that this distribution
is {\em orthogonally invariant} in the sense that $Z_n$ follows the same
distribution as $A Z_n B$, for every orthogonal $A\in M_{m_n,m_n}$ and $B\in
M_{n,n}$. This is the case, for example, when the entries of $Z_n$ are
Gaussian.  In Section \ref{general_noise:sec}   we revisit this restriction
and discuss general (not necessarily invariant) white noise.

\item {\em Fixed signal column span $(x_1,\ldots,x_r)$:}
Let the rank $\rk>0$ be fixed and choose a vector $\V{x}\in\R^\rk$ with
coordinates $\V{x}=(x_1,\ldots,x_\rk)$ such that $x_1> \ldots > x_r>0$. Assume that for all $n$, 
\begin{eqnarray} \label{singvec:eq}
X_n = U_n \, diag(x_1,\ldots,x_r,0,\ldots,0) \, \Ut_n'\,
\end{eqnarray}
is an arbitrary
 singular value decomposition of $X_n$, where $U_n\in M_{m_n,m_n}$ and 
$\Ut_n\in M_{n,n}$. 

\item {\em Asymptotic aspect ratio $\beta$:}
    The sequence $\m_\n$ is such that $\m_\n / \n \to \beta$. To simplify our
formulas,
we assume that $0< \beta\leq 1$. 

\end{enumerate}

Note that while the signal rank $r$ and nonzero signal singular values 
$x_1,\ldots,x_r$ are shared by all matrices $X_n$, the signal left and right
singular vectors $U_n$ and $V_n$ are unknown and arbitrary.
We also remark that the assumption, whereby the signal singular values are non-degenerate
  ($x_i>x_{i+1}$, $1\leq i < r$), is not
  necessary for our results to hold, yet it simplifies the analysis
  considerably.

\begin{defn} \label{asyloss:def} {\bf Asymptotic Loss.}
  Let $L=\{L_{m,n}\,|\, (m,n)\in\mathbb{N}\times\mathbb{N}\}$ be a family of losses, where each $L_{m,n}:
\Mmn\times \Mmn\to[0,\infty)$ is a loss function obeying
    $L_{m,n}(X,X)=0$. 
   Let $\eta:[0,\infty)\to[0,\infty)$ be a
     nonlinearity and consider 
   $\hat{X}_\eta$, the singular value shrinkage denoiser \eqref{shrink:eq}
   calibrated, as discussed above, for noise level $1/\sqrt{n}$. Let $m_n$ be an
   increasing sequence such that $\lim_{n\to\infty} m_n/n = \beta$, implicit in
   our notation. 
   Define the {\em asymptotic loss} of the shrinker $\eta$  (with respect to
   $L$)
   at the signal $\V{x}=(x_1,\ldots,x_r)$ by 
   \[
     L_\infty(\eta|\V{x}) \aseq \lim_{n\to\infty} L_{m_n,n}\left(X_n\,,\,\hat{X}_\eta(X_n +
     \tfrac{1}{\sqrt{n}}Z_n)\right)\,
   \]
   when the limit exists.
 \end{defn}

Our results imply that the asymptotic loss $L_\infty$ exists and is well-defined, as a function of the signal
singular values $\V{x}$, for a large class of nonlinearities.

\begin{defn} \label{optShrink:def}
  {\bf Optimal Shrinker.}
Let $L$ be a loss family. If a 
shrinker $\eta^*$ has an asymptotic loss that satisfies 
\[
  L_\infty(\eta^*|\V{x}) \leq L_\infty(\eta|\V{x})
\]
for any other 
shrinker $\eta$ in a certain class of shrinkers, any $r\geq 1$ and any $\V{x}\in\R^r$,
then we say that $\eta^*$ is {\em unique asymptotically admissible}
(of simply ``optimal'')
for the loss sequence $L$ and that class of shrinkers.
\end{defn}

\subsection{Our contribution}

At first glance, it seems too much to hope that optimal shrinkers in the sense of Definition
\ref{optShrink:def} even exist. Indeed, existence of an optimal shrinker for
a loss family $L$ implies that, asymptotically, the decision-theoretic picture
is extremely simple and actionable: from the asymptotic loss perspective, there is a single rational choice for shrinker. 

In our current terminology, Shabalin and
Nobel   \cite{Shabalin2013} have effectively shown that an optimal shrinker
exists for Frobenius loss.
The estimator they derive can be shown to be 
equivalent to the optimal shrinker \eqref{opt_shrinker_fro:eq}, yet was given in
a more complicated form. (In Section \ref{fro:sec} we visit the special case of Frobenius
loss in detail, and prove that \eqref{opt_shrinker_fro:eq} is the optimal
shrinker.) 

Our contribution in this paper is as follows.

\begin{enumerate}
  \item 
   We rigorously establish
the existence of an optimal shrinker for a variety of loss families, 
including the popular Frobenius, operator and nuclear norm losses. 

\item 
  We provide a framework for finding the
  optimal shrinkers for a variety of loss families including these popular
  losses. As discussed in Section \ref{noise:sec}, our framework can be applied 
  whether the noise level $\sigma$ is
  known or unknown.

\item 
  We use
  our framework to find simple, explicit formulas for the optimal
  shrinkers for  Frobenius, operator and nuclear norm losses, and show that it
  allows simple numerical evaluation  of optimal shrinkers when  a closed-form
  formula for the optimal shrinker is unavailable.

\end{enumerate}

In the related problem of covariance estimation in the Spiked
Covariance Model, in collaboration with I. Johnstone we identified a similar
phenomenon, namely, existence of optimal {\em eigenvalue shrinkers}
 for covariance estimation \cite{Donohoa}.

\section{The Asymptotic Picture}

In the ``null case'' $X_n \equiv 0$, the empirical distribution of the singular
values of $Y_n=Z_n/\sqrt{n}$  famously converges as $n\to\infty$ to the
generalized quarter-circle
distribution \cite{Anderson2010}, whose density is
\begin{eqnarray} \label{mp:eq}
  f(x) = \frac{\sqrt{4\beta-(x^2-1-\beta)^2}}{\pi \beta
  x}\mathbf{1}_{[1-\sqrt{\beta},1+\sqrt{\beta]}}(x)\,.
\end{eqnarray}
 This 
distribution is compactly supported on $[\beta_-,\beta_+]$, with 
\[
  \beta_\pm = 1\pm \sqrt{\beta}\,.
\]
Moreover, in this null case we have  $y_{n,1}\aslim 1+\sqrt{\beta}$, see 
\cite{Yin1988,Bai1993}.
We say that the singular values of $Y_n$ form a (generalized) quarter circle
{\em bulk} and call $\beta_+$ the {\em bulk edge}.

Expanding seminal results of \cite{Paul2007,Dozier2007} and many other authors,
 Benaych-Georges and Nadakuditi \cite{Benaych-Georges2012} have provided
a thorough analysis of a collection of models, which includes the model
\eqref{model:eq} as a special case. In this section we summarize some of their
results regarding asymptotic behaviour of the model \eqref{model:eq}, which are
relevant to singular value shrinkage.

\begin{figure}[h!]
  \centering
  \includegraphics[height=2.3in]{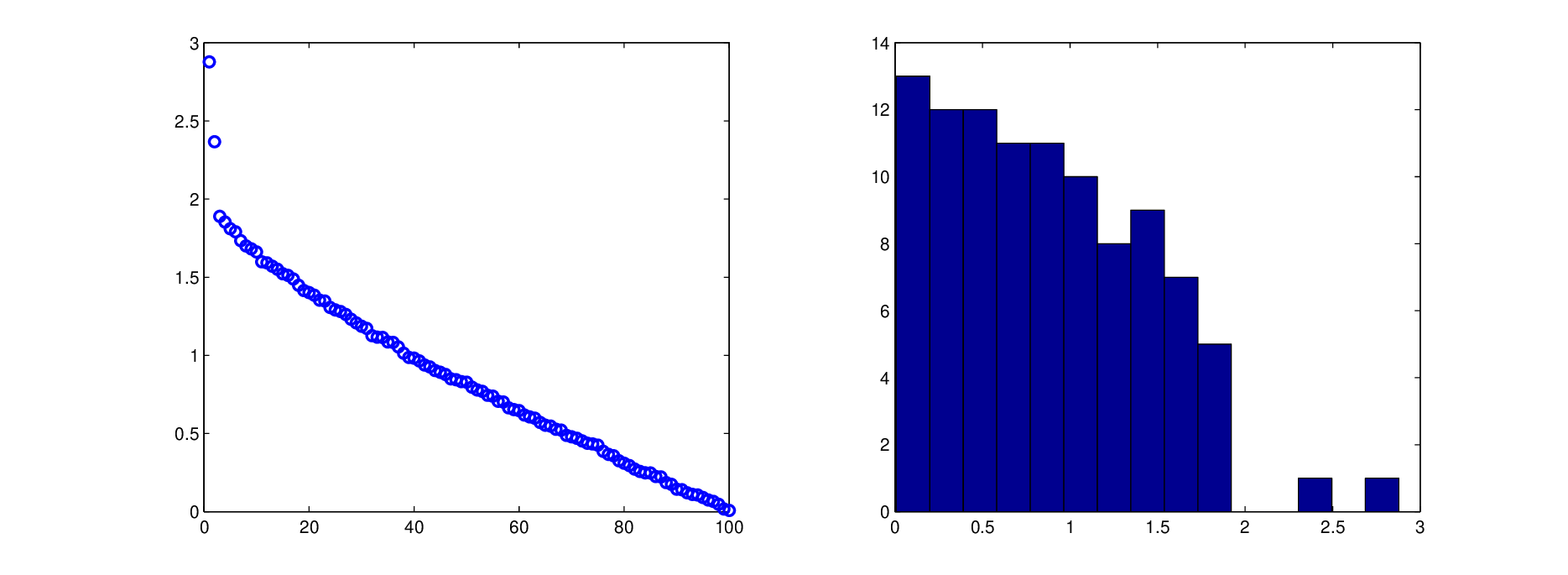}
  \caption{\small Singular values of a data matrix $Y\in M_{100,100}$ drawn from the
  model $Y=X+Z/\sqrt{100}$, with $r=2$ and $\V{x}=(2.5,1.7)$. Left: singular
  values in decreasing order. Right: Histogram of the singular values (note the bulk edge
close to $2$). This figure can be reproduced
 using the code supplement \cite{Gavish2015}.}
  \label{bulk:fig}
\end{figure}

For $x\geq \beta^{1/4}$, define
\begin{eqnarray}
  y(x) &=&  \sqrt{\left(x + \frac{1}{x}\right)\left(x +
  \frac{\beta}{x}\right)}\,,  \label{y:eq} \\
  c(x) &=&  \sqrt{\frac{x^4 - \beta}{x^4 + \beta x^2}} \qquad \text{and} \label{c:eq}\\
  \ct(x) &=&   \sqrt{\frac{x^4 - \beta}{x^4 +  x^2}} \,. \label{ct:eq}
\end{eqnarray}

It turns out that  $y(x)$ from Eq. \eqref{y:eq} is the
asymptotic location of a data singular value corresponding to a signal singular
value $x$, provided $x\geq\beta^{1/4}$. (Note that the function $x(y)$ from
\eqref{x_of_y:eq} is the inverse of $y(x)$ when $x\geq \beta^{1/4}$, and that
$y(\beta^{1/4})=\beta_+$.) Similarly, $c(x)$ from Eq. \eqref{c:eq}
(resp.  $\ct(x)$ from Eq. \eqref{ct:eq}) is the
cosine of the asymptotic angle between the signal left (resp. right) singular
vector and the corresponding data left (resp. right) singular vector, provided
that the corresponding signal singular value $x$ satisfies $x\geq\beta^{1/4}$.

Additional notation is required to state these facts formally. We rewrite the sequence of signal
matrices in our asymptotic framework \eqref{singvec:eq} as
\begin{eqnarray} \label{Xn:eq}
  X_n=\sum_{i=1}^r x_i \, \V{u}_{n,i} \, \V{\ut}_{n,i}'\,,
\end{eqnarray}
 so that $\V{u}_{n,i}\in\R^{m_n}$ (resp.
$\V{\ut}_{n,i}\in\R^n$) is the
left (resp. right) singular vector corresponding to the singular value $x_i$,
namely, 
$i$-th column of $U_n$ (resp. $\Ut_n$) in \eqref{singvec:eq}.
Similarly, let $Y_\n$ be a corresponding sequence of observed matrices in our
framework, and write 
\begin{eqnarray} \label{Yn:eq}
  Y_n = \sum_{i=1}^{m_n} y_{n,i} \,\V{v}_{n,i}\, \V{\vt}_{n,i}'
\end{eqnarray}
so that $\V{v}_{n,i}\in\R^m$ (resp. $\V{\vt}_{n,i}\in\R^n$) is the
left (resp. right) singular vector corresponding to the singular value $y_{n,i}$.
\\~\\
In our notation,  Lemma \ref{y-asy:lem} and Lemma \ref{inner-asy:lem} follow
from  Theorem 2.9 and Theorem 2.10 of \cite{Benaych-Georges2012}:

\begin{lemma} \label{y-asy:lem}
  {\bf Asymptotic location of the top $r$ data singular values.}
  For $1\leq i\leq \rk$,
  \begin{eqnarray} 
    \lim_{\n\to\infty} y_{\n,i} \aseq
    \begin{cases}
    y(x_i) 
    & x_i \geq \beta^{1/4} \\
    \beta_+ & x_i<\beta^{1/4}
  \end{cases}\,.
  \end{eqnarray}
\end{lemma}

  \begin{lemma} \label{inner-asy:lem}
  {\bf Asymptotic angle between signal and data singular vectors.}
  Let $1\leq i\neq j\leq \rk$ and assume that $x_i\geq \beta^{1/4}$ is
  non-degenerate, namely, the value $x_i$ appears only once in $\V{x}$.
Then
  \begin{eqnarray} \label{proj-u:eq}
  \lim_{\n\to\infty} \big|\langle \V{u}_{n,i}\,,\,\V{v}_{\n,j}\rangle\big| \stackrel{a.s.}{=} 
    \begin{cases}
c(x_i) & i = j \\
      0 & i \neq j
    \end{cases}\,,
  \end{eqnarray}
and
\begin{eqnarray} \label{proj-v:eq}
   \lim_{\n\to\infty} \big|\langle \V{\ut}_{n,i}\,,\,\V{\vt}_{\n,j}\rangle\big| \stackrel{a.s.}{=} 
    \begin{cases}
      \ct(x_i)  & i = j \\
      0 & i \neq j
    \end{cases}\,.
  \end{eqnarray}
  If however $x_i<\beta^{1/4}$, then we have
  \begin{eqnarray*}
    \lim_{\n\to\infty} \big|\langle \V{u}_{n,i}\,,\,\V{v}_{\n,j}\rangle\big| \stackrel{a.s.}{=} 
    \lim_{\n\to\infty} \big|\langle \V{\ut}_{n,i}\,,\,\V{\vt}_{\n,j}\rangle\big| \stackrel{a.s.}{=} 
0\,.
  \end{eqnarray*}
\end{lemma}
We also note the following fact regarding the data singular values
\cite[proof of Theorem 2.9]{Benaych-Georges2012}:
\begin{lemma} \label{fixedrank:lemma}
  Let $i>r$ be fixed.  Then $y_{n,i}\aslim \beta_+$.
\end{lemma}
\section{Optimal Shrinker for Frobenius Loss} \label{fro:sec}

As an introduction to the more general framework developed below, we first
examine the Frobenius loss case, following the work of Shabalin and Nobel
\cite{Shabalin2013}. Using Definition \ref{asyloss:def}, 
let $L=\left\{ L_{m,n} \right\}$ be the Frobenius loss
family, namely $L_{m,n}$ is given by \eqref{eq:fro}. 

\subsection{Lower bound on asymptotic loss}

Directly expanding the Frobenius matrix norm, we obtain:

\begin{lemma} \label{MSE:lem} {\bf Frobenius loss of singular value shrinkage.}
  For any shrinker $\eta:[0,\infty)\to[0,\infty)$, we have
    \begin{eqnarray}
        \norm{X_n-\hat{X}_\eta(X_n + Z_n/\sqrt{n})}_F^2 &=&  \sum_{i=1}^r 
        \left[ x_i^2 + (\eta(y_{n,i}))^2  \right] \label{MSE1:eq} \\
            &-&
              2\sum_{i,j=1}^r x_i\eta(y_{n,i})
                \inner{\V{u}_{n,i}}{\V{v}_{n,j}}\inner{\V{\ut}_{n,i}}{\V{\vt}_{n,j}}
                \label{MSE2:eq} \\
                &+&  \sum_{i=r+1}^{m_n}
                (\eta(y_{n,i}))^2 \label{residual:eq} 
              \end{eqnarray}
\end{lemma}
~
\\
This implies a lower bound on Frobenius loss of any singular value shrinker: 

\begin{cor} \label{lower_bound_MSE:cor}
  For any shrinker $\eta:[0,\infty)\to[0,\infty)$, we have
    \begin{eqnarray*}
        \norm{X_n-\hat{X}_\eta(X_n + Z_n/\sqrt{n})}_F^2 &\geq&  \sum_{i=1}^r 
        \left[ x_i^2 + (\eta(y_{n,i}))^2  \right] \label{MSE:eq}\\ 
          &-&
              2\sum_{i,j=1}^r x_i\eta(y_{n,i})
                \inner{\V{u}_{n,i}}{\V{v}_{n,j}}\inner{\V{\ut}_{n,i}}{\V{\vt}_{n,j}}\,.
              \end{eqnarray*}
\end{cor}

As $n\to\infty$, this lower bound on the Frobenius loss is governed by 
three quantities: the asymptotic location of data singular value
$y_{n,i}$, the asymptotic angle between the left signal singular vectors 
and left data singular vector
$\inner{\V{u}_{n,i}}{\V{v}_{n,i}}$, and asymptotic angle between the right
signal singular vectors and right data singular vector
$\inner{\V{\tilde{u}}_{n,i}}{\V{\tilde{v}}_{n,i}}$
(see also \cite{Donoho2013b}).

Combining  Corollary \ref{lower_bound_MSE:cor}, Lemma \ref{y-asy:lem}
and Lemma \ref{inner-asy:lem} we obtain a lower bound for the asymptotic  Frobenius
loss (see \cite{Shabalin2013}):

\begin{cor} \label{lowerbound:cor}
 For any 
 continuous 
 shrinker $\eta:[0,\infty)\to[0,\infty)$, we have
   \begin{eqnarray*}
     L_\infty(\eta|\V{x}) \aseq \lim_{n\to\infty} 
     \norm{X_n-\hat{X}_\eta(X_n + Z_n/\sqrt{n})}_F^2 \geq  
     \sum_{i=1}^r  L_{2,2}(\eta|x_i)
              \end{eqnarray*}
where
\begin{eqnarray} \label{L22_surprise:eq}
  L_{2,2}(\eta|x) =  x^2 + \eta^2 
          - 2 x \eta c(x)\ct(x)\,.
        \end{eqnarray}
and $\eta = \eta(y)$.
\end{cor}
\noindent
The notation $L_{2,2}$  in \eqref{L22_surprise:eq} will be made apparent below,
see \eqref{L22_fro:eq}.

\subsection{Optimal shrinker matching the lower bound}

By differentiating the asymptotic lower bound w.r.t $\eta$, we find that
$L_\infty(\eta|\V{x}) \geq \sum_{i=1}^r L_{2,2}(\eta^*|x_i)$, where 
$\eta^*(y(x)) = xc(x)\ct(x)$.
Expanding $c(x)$ and $\ct(x)$ from Eqs. \eqref{c:eq} and \eqref{ct:eq}, we
find that $\eta^*(y)$ is given by \eqref{opt_shrinker_fro:eq}. 

The singular value shrinker  $\eta^*$, for which $\hat{X}_{\eta^*}$ minimizes the
asymptotic lower bound, thus becomes a natural
candidate for the optimal shrinker for Frobenius loss. 
Indeed, by definition, for $\hat{X}_{\eta^*}$ the limits of \eqref{MSE1:eq} and \eqref{MSE2:eq}
are the smallest possible. It remains to show that the limit of
\eqref{residual:eq} is the smallest possible. 

It is clear from \eqref{residual:eq} that a necessary condition for a shrinker $\eta$ to be
successful, let alone optimal, is that it must set to zero data eigenvalues that do not correspond to
signal. 
With \eqref{residual:eq} in mind, 
we should only consider shrinkers $\eta$ for which $\eta(y)=0$ for any $y\leq
\beta_+$.  
The following is a sufficient condition for a shrinker to achieve the
lowest limit possible in the term \eqref{residual:eq}, namely, for this term to 
converge to zero.

\begin{defn} \label{conservative:def}
  Assume that a continuous shrinker  $\eta:[0,\infty)\to[0,\infty)$  
satisfies $\eta(y)=0$ whenever $y\leq\beta_++\varepsilon$ \,for some 
fixed $\varepsilon>0$.
    We say that $\eta$ is a  {\em Conservative shrinker}.
\end{defn}

By Lemma
\ref{y-asy:lem} and Lemma \ref{inner-asy:lem}, it is clear that 
conservative shrinkers set to zero all data singular
values $\left\{ y_i \right\}$ which originate from pure noise ($x_i=0$), as well as 
all data singular values $\left\{ y_i \right\}$ which are ``engulfed'' in the
noise bulk, rendering their corresponding singular vectors useless 
($x_i < 
\beta^{1/4}$). Conservative shrinkers are so called since they 
leave a (possibly infinitesimally small)
safety margin $\varepsilon$. They enjoy the following key property:

\begin{lemma} \label{residual:lem}
  Let  $\eta:[0,\infty)\to[0,\infty)$ be a conservative shrinker.
     Then 
    \[
       \sum_{i=r+1}^{m_n}
                (\eta(y_{n,i}))^2 \aslim 0\,.
    \]
\end{lemma}
\begin{proof}
  By Lemma \ref{fixedrank:lemma} we have $y_{n,r+1}\aslim \beta_+$. Let $N$ be
  the
  (random) index such that $y_{n,r+1}<\beta_+ +\varepsilon$ for all $n>N$. 
  Then for all $n>N$ and all $i>r$ we have $y_{n,i}<\beta_+ + \varepsilon$,
  hence $\eta(y_{n,i})=0$. The desired almost sure convergence follows.
\end{proof}

Ironically, careful inspection of the candidate \eqref{opt_shrinker_fro:eq} reveals that it is 
continuous yet
not strictly conservative: it only satisfies $\eta(y)=0$ for $y\leq \beta_+$,
leaving no margin above the bulk edge $\beta_+$.
In fact, building on Lipschitz continuity of the Frobenius
loss itself, it can be shown 
that Lemma \ref{residual:lem}  remains true for the shrinker
\eqref{opt_shrinker_fro:eq}  as well \cite{Donohoa}; this is however outside our 
present scope. Consequently, the asymptotic loss of \eqref{opt_shrinker_fro:eq} 
matches the
lower bound from Corollary \ref{lowerbound:cor}, and is lower than the
asymptotic loss of any other continuous shrinker,  for any low-rank model 
$(x_1,\ldots,x_r)$.

\section{A Framework for Finding Optimal Shrinkers} \label{framework:sec}

With the previous section in mind, 
our main result may be summarized as follows:  the basic ingredients that enabled us
to find the optimal shrinker for Frobenius loss allow us 
to find the optimal shrinker for each of a  variety of loss families.
For these loss families, an optimal shrinker exists and is given by a simple
formula. To avoid some technical nuisance, we focus on finding the optimal
shrinker among conservative shrinkers.

To get started, let us describe the loss families to which our method applies.

\begin{defn} {\bf Orthogonally invariant loss.}
  A loss $L_{m,n}(\cdot,\cdot)$ is {\em orthogonally invariant} if for all
  $m,n$ we have
  $L_{m,n}(A,B) = L_{m,n}(UAV,UBV) $, for any orthogonal $U\in O_m$ and $V\in O_n$.
\end{defn}

\begin{defn} {\bf Decomposable loss family.}
  Let $A,B\in\Mmn$ and let $m=\sum_{i=1}^k m_i$ and $n=\sum_{i=1}^k n_i$. 
Assume that there are matrices $A_i,B_i\in M_{m_i,n_i}$, $1\leq i\leq k$, such
that
 \[A=\oplus_i A_i \qquad \qquad B=\oplus_i B_i\] 
 in the sense that $A$ and $B$ are block-diagonal with blocks $\left\{ A_i
  \right\}$ and  $\left\{ B_i \right\}$, respectively.
  A loss family $L=\{L_{m,n}\}$ is {\em sum-decomposable} if, for all $m,n$ and
   $A,B$ with block diagonal structure as above,
   \[L_{m,n}(A,B)=\sum_i L_{m_i,n_i}(A_i,B_i)\,.\] 
Similarly, it is {\em max-decomposable} if   \[L_{m,n}(A,B)=\max_i
L_{m_i,n_i}(A_i,B_i)\,.\] 
\end{defn}

\paragraph{Examples.} 
As primary examples, we consider loss families defined in Section
\ref{some:subsec}: The Frobenius norm loss $L^{fro}$, the operator norm loss
$L^{op}$ and the nuclear norm loss $L^{nuc}$. 
It is easy to check that (i) each of these losses are orthogonally invariant, 
and (ii) the families $L^{fro}$ and $L^{nuc}$ are sum-decomposable, while the family
$L^{op}$ is max-decomposable.
\\
~
\\
Our framework for finding optimal shrinkers can now be stated as  follows.

\begin{thm}  \label{char:thm}
  {\bf Characterization of the optimal singular value shrinker.}
 Let 
\begin{eqnarray}
A(x) &=&   \begin{bmatrix} x & 0 \\ 0 & 0 \end{bmatrix} \\
 B(\eta,x) &=&  
\eta \begin{bmatrix}  c(x)\,\ct(x)  &
  c(x)\,\stt(x)  \\
  \ct(x)\, s(x) & s(x)\,\stt(x) 
\end{bmatrix}\,, \label{B:eq}
\end{eqnarray}
where $c(x)$ and $\ct(x)$ are given by Eqs. \eqref{c:eq} and \eqref{ct:eq}, and
where $s(x)=\sqrt{1-c^2(x)}$ and $\stt(x)=\sqrt{1-\ct^2(x)}$.
Assume that $L=\left\{ L_{m,n} \right\}$ is a sum- or max- decomposable family
of orthogonally invariant losses. %
Define 
\begin{eqnarray}\label{F:eq}
  F(\eta,x) = L_{2,2}(A(x),B(\eta,x))
\end{eqnarray}
and suppose that for any $x\geq\beta^{1/4}$ there exists a unique minimizer
\begin{eqnarray} \label{Fmin:eq}
  \eta^{**}(x) = argmin_{\eta\geq0} F(\eta,x)\,,
\end{eqnarray}
such that $\eta^{**}$ is a conservative shrinker on $[\beta^{1/4},\infty)$.
Further suppose that there exists a point $x_0 \geq \beta^{1/4}$ such that 
\[
  F(\eta^{**}(x),x) \geq L_{1,1}(x,0)\qquad\qquad \beta^{1/4}\leq x\leq x_0\,,
\]
with
\begin{eqnarray} \label{cross:eq}
  F(\eta^{**}(x_0),x_0) = L_{1,1}(x_0,0)\,.
\end{eqnarray}
Define the shrinker 
\begin{eqnarray} \label{concat:eq}
  \eta^*(y) = 
  \begin{cases}
    \eta^{**}(x(y)) & y(x_0)\leq y  \\
    0 & 0\leq y < y(x_0)
\end{cases}\,,
\end{eqnarray}
where $x(y)$ is defined in Eq.
\eqref{x_of_y:eq}. 
Then for any conservative shrinker $\eta$, 
the asymptotic losses $L_\infty(\eta^*|\cdot)$ and $L_\infty(\eta|\cdot)$ exist,
and
\[
  L_\infty(\eta^*|\V{x})\leq L_\infty(\eta|\V{x})
\]
for all $r\geq 1$ and all $\V{x}\in\R^r$.
\end{thm}

\subsection{Discussion} \label{discuss:subsec}

Before we proceed to prove Theorem \ref{char:thm}, we review the information it
encodes about the problem at hand
and its operational meaning.
Theorem \ref{char:thm} is based on a few simple observations:
\begin{itemize}
  \item First,
if 
$L$ is a sum-- (resp. max--)
decomposable family of orthogonally invariant losses, and if $\eta$ is
a conservative shrinker, then the asymptotic loss $L_\infty(\eta|\V{x})$ at
$\V{x}=(x_1,\ldots,x_r)$ can be written as a sum (resp. a maximum) over $r$
terms.
These terms
have identical functional form. When $x_i\geq \beta^{1/4}$, these terms have the
form
$L_{2,2}(A(x_i),B(\eta,x_i))$, and when $0\leq x_i < \beta^{1/4}$, these terms
have the form $L_{1,1}(x_i,0)+L_{1,1}(0,\eta)$  (resp. 
$\max\{L_{1,1}(x_i,0)\,,\,L_{1,1}(0,\eta)\}$).
 As a result, one finds that the zero shrinker $\eta\equiv 0$ is necessarily optimal for
 $0\leq x\leq \beta^{1/4}$. For $x\geq \beta^{1/4}$, one just needs to
minimize the loss of a specific $2$-by-$2$ matrix, namely the function $F$ from
\eqref{F:eq}, to obtain the shrinker $\eta^{**}$ of \eqref{Fmin:eq}. 

\item Second, the asymptotic loss curve 
  $L_\infty(\eta^{**}|x)$ necessarily crosses the asymptotic loss curve of
  the zero shrinker $L_\infty(\eta\equiv 0|x)$ at a point we will denote by 
  $x_0$, with $x_0\geq \beta^{1/4}$.

\item Finally, by concatenating the zero shrinker and the shrinker $\eta^{**}$ 
  precisely at the point $x_0$ where their asymptotic losses cross, one obtains
  a shrinker which is continuous ($x_0>\beta^{1/4}$) or possibly
  discontinuous ($x_0=\beta^{1/4}$).  However, this shrinker always has a well-defined
  asymptotic loss. This loss dominates the asymptotic loss of any conservative shrinker.

\end{itemize}

For some loss families  $L=\left\{ L_{m,n} \right\}$, it is possible to find an
explicit formula for the optimal shrinker using the following steps:
\begin{enumerate}
  \item Write down an explicit expression for the function $F(\eta,x)$ from
    \eqref{F:eq}.
  \item Explicitly solve for the minimizer $\eta^{**}(x)$ from \eqref{Fmin:eq}.
  \item Write down an explicit expression for the minimum $F(\eta^{**}(x),x)$. 
  \item Solve \eqref{cross:eq} for the crossing point $x_0$.
  \item Compose $\eta^{**}(x)$ with the transformation $x(y)$ from
    \eqref{x_of_y:eq} to obtain an explicit form of the optimal shrinker
    $\eta^*(y)$ from \eqref{concat:eq}.
\end{enumerate}

In Sections \ref{froopnuc:sec} and \ref{schatten:sec} we offer examples of this
process: in Section \ref{froopnuc:sec}  we follow it analytically and derive
simple, explicit formulae of
the optimal shrinkers for the Frobenius, operator and nuclear norm losses. In
Section \ref{schatten:sec} we follow it numerically and compute the optimal
shrinker for any Schatten-$p$ norm loss.

In the remainder of this section we describe a sequence of constructions and
lemmas leading to the proof of Theorem \ref{char:thm}. 

  \subsection{Simultaneous Block Diagonalization} \label{block:subsec}

Let us start by considering a fixed signal matrix and noise matrix,
without placing them in a sequence. To allow a gentle exposition of the main
ideas, we initially make two simplifying assumptions: first, that $r=1$, namely
that $X$ is rank-$1$, and second, that $\eta$ shrinks to zero all but the first
singular values of $Y$, namely, $\eta(y_i)=0$, $i>1$.
Let $X\in\Mmn$ be a signal
matrix and let  
 $Y=X+Z/\sqrt{n}\in \Mmn$ be a corresponding data matrix. 
 Denote their SVD by 
 \[ X=U \cdot
diag(x_1,0,\ldots,0) \cdot \Ut \qquad \qquad  Y=V\cdot
diag(y_1,\ldots,y_m)\cdot \Vt\,\,.
\]
 Write $0_{m\times n}$ for the $m$-by-$n$
matrix whose entries are all zeros. The basis pairs  $U,\Ut$ and $V,\Vt$
diagonalize $X$ and $Y$, respectively. Indeed, since $\eta(y_i)=0$ for $i\geq 2$, we have
\begin{eqnarray*}
  X = x_1 \V{u}_1 \V{\tilde{u}}_1' 
  \qquad \qquad 
  \hat{X}_\eta(Y) = \eta(y_1) \V{v}_1 \V{\tilde{v}}_1' \,\,.
\end{eqnarray*}

Combining the basis pairs $U,\tilde{U}$ and $V,\tilde{V}$, 
we are lead to the following ``common''
basis pair, which we will denote by $W,\Wt$. Let $w_1,\ldots,w_m$ denote the
orthonormal basis constructed by applying the Gram-–Schmidt process to the
sequence $\V{u}_1,\V{v}_1,\V{v}_2,\ldots,\V{v}_{m-1}$, where $\V{u}_i$ is the
$i$-th column of $U$, namely the $i$-th left singular vector of $X$, and
similarly, $\V{v}_i$ is the $i$-th column of $V$. Denote by $W$ the matrix whose
columns are $\V{w}_1,\ldots,\V{w}_m$. Repeating this construction for $\Ut$ and
$\Vt$, let $\wt_1,\ldots,\wt_n$ denote the
orthonormal basis constructed by applying the Gram-Schmidt process to the
sequence $\V{\ut}_1,\V{\vt}_1,\V{\vt}_2,\ldots,\V{\vt}_{m-1}$, where $\V{\ut}_i$ is the
$i$-th column of $U$, namely the $i$-th right singular vector of $X$, and
similarly, $\V{\vt}_i$ is the $i$-th column of $\Vt$. Denote by $\Wt$ the matrix whose
columns are $\V{\wt}_1,\ldots,\V{\wt}_m$.

Specifically, if $\V{v}_1$ and $\V{u}_1$ are not colinear, we let
$\V{w}_1=\V{u}_1$ and let $\V{w}_2$ be the
first Gram-Schmidt step for the sequence
$\V{u}_1,\V{v}_1,\V{v}_2,\ldots,\V{v}_{m-1}$, namely, 
we let $\V{w}_2 = s_1^{-1}(\V{v}_1 - c_1\V{u}_1)$, 
where $c_1=\inner{\V{u}_1}{\V{v}_1}$ and $s_1=\sqrt{1-c_1^2}$. If it happens
that  $\V{v}_1$ and $\V{u}_1$ are colinear, we let $\V{w}_2$ be any vector
orthogonal to $\V{u}_1$, for example $\V{u}_m$. The rest of the Gram-Schmidt
process proceeds to add $m-2$ additional unit vectors orthogonal to both
$\V{w}_1$ and $\V{w}_2$. Observe that $Span\left\{ \V{u}_1,\V{v}_1
\right\}\subset Span\left\{ \V{w}_1,\V{w}_2 \right\}$, so that
$W'\V{u}_1=(1,0,\ldots,0)'$ and $W'\V{v}_1=(c_1,s_1,0,\ldots,0)'$.
Repeating the same construction for the right singular vector basis $\tilde{U}$ and
$\tilde{V}$, we obtain the basis $\tilde{W}$ with the property that
$\Wt'\V{\tilde{u}}_1=(1,0,\ldots,0)'$ and
$\Wt'\V{\tilde{v}}_1=(\tilde{c}_1,\tilde{s}_1,0,\ldots,0)'$.

Writing now $X$ and $\hat{X}_\eta$ in our new basis pair $W,\Wt$ we get
\begin{eqnarray}
  W'\, X \,\Wt &=& x_1\, (W' \V{u}_1)  (\Wt'\V{\tilde{u}}_1)'  = \begin{bmatrix} x_1 & 0 \\ 0 & 0 \end{bmatrix} \oplus 0_{m-2\times n-2}   \\
W'\, \hat{X}_\eta(Y)\, \Wt  &=& 
\eta(y_1) \, (W' \V{v}_1) (\Wt' \V{\tilde{v}}_1)' =
\eta(y_1)
\begin{bmatrix}  c_1\ct_1  & c_1\stt_1  \\
  \ct_1 s_1 & s_1\stt_1 
\end{bmatrix} 
\oplus 0_{m-2\times n-2}  \,.
\end{eqnarray}

It is convenient to rewrite this as
\begin{eqnarray}
 W'\, X \, \Wt &=&  A(x_1) \oplus 0_{m-2\times n-2} \\
 W'\, \hat{X}_\eta(Y)\, \Wt &=&  B(\eta(y_1),c_1,s_1,\ct_1,\stt_1)
\oplus 0_{m-2\times n-2}
\end{eqnarray}
where
\begin{eqnarray}
A(x) &=&   \begin{bmatrix} x & 0 \\ 0 & 0 \end{bmatrix} \\
 B(\eta,c,s,\ct,\stt) &=&  
\eta \begin{bmatrix}  c\ct  & c\stt  \\
  \ct s & s\stt 
\end{bmatrix} \,.
\end{eqnarray}

Thus, if $L=\left\{ L_{m,n} \right\}$ is a sum- or max-decomposable family of
orthogonally invariant functions, we have
\begin{eqnarray*}
  L_{m,n}(X\,,\,\hat{X}_\eta(Y)) &=& 
  L_{m,n}(W'X\Wt\,,\,W'\hat{X}_\eta(Y)\Wt) \\
  &=& L_{m,n}\Big(A(x_1)\oplus 0_{m-2\times
  n-2}\,,\,B(\eta(y_1),c_1,s_1,\ct_1,\stt_1)\oplus
  0_{m-2\times n-2}\Big) \\
  &=&  L_{2,2}\Big(A(x_1)\,,\,B(\eta(y_1),c_1,s_1,\ct_1,\stt_1)\Big)\,.
\end{eqnarray*}

We have proved:
\begin{lemma} \label{block1:lem}
  Let $X=x_1\V{u}_1\V{\ut}_1'\in\Mmn$ be rank-$1$ and assume that $Y=\sum_{i=1}^m
  y_i \V{v}_i\V{\vt}_i'$ and $\eta$ are such that
  $\eta(y_i)=0$, $i>1$, where $y_i$ is the $i$-th largest singular value of $Y$.
  Let $L=\left\{ L_{m,n} \right\}$ be a sum- or max-decomposable
  orthogonally invariant loss family.
  Then   
\[ 
  L_{m,n}(X\,,\,\hat{X}_\eta(Y)) 
= L_{2,2}\Big(A(x_1)\,,\,B(\eta(y_1),c_1,s_1,\ct_1,\stt_1)\Big)\,,\]
where
\begin{eqnarray*}
  c_1=\inner{\V{u}_1}{\V{v}_1} && s_1=\sqrt{1-c_1^2} \\
  \ct_1=\inner{\V{\ut}_1}{\V{\vt}_1} && \stt_1=\sqrt{1-\ct_1^2} \,.
\end{eqnarray*}
\end{lemma}

A similar argument gives a similar statement for rank-$r$ matrix
$X$ with non-degenerate singular values:

\begin{lemma} {\bf A decomposition for the loss.} \label{blockmany:lem}
  Let $X=\sum_{i=1}^r x_i\V{u}_i\V{\ut}_i'\in\Mmn$ be rank-$r$ with
  $x_1>\ldots>x_r>0$, and assume that $Y=\sum_{i=1}^m
  y_i \V{v}_i\V{\vt}_i'$ and $\eta$ are such that
  $\eta(y_i)=0$, $i>r$, where $y_i$ is the $i$-th largest singular value of $Y$.
  Let $L=\left\{ L_{m,n} \right\}$ be a sum- or max-decomposable family of
  orthogonally invariant functions.
  Then   
\[ 
  L_{m,n}(X\,,\,\hat{X}_\eta(Y)) =
\sum_{i=1}^r
 L_{2,2}\Big(A(x_i)\,,\,B(\eta(y_i),c_i,s_i,\ct_i,\stt_i)\Big)\,,\]
if $L$ is sum-decomposable, or 
\[ 
  L_{m,n}(X\,,\,\hat{X}_\eta(Y)) =
\max_{i=1}^r
 L_{2,2}\Big(A(x_i)\,,\,B(\eta(y_i),c_i,s_i,\ct_i,\stt_i)\Big)\,,\]
if $L$ is max-decomposable.
Here, 
\begin{eqnarray*}
  c_i=\inner{\V{u}_i}{\V{v}_i} && s_i=\sqrt{1-c_i^2} \\
  \ct_i=\inner{\V{\ut}_i}{\V{\vt}_i} && \stt_i=\sqrt{1-\ct_i^2} \,,
\end{eqnarray*}
for $i=1,\ldots,r$.
\end{lemma}

See \cite{Donohoa} for the complete proof. In short, we apply the Gram-Schmidt
process to the vector sequences
$\V{u}_1,\ldots,\V{u}_r,\V{v}_1,\ldots,\V{v}_{m-r}$ and 
$\V{\tilde{u}}_1,\ldots,\V{\tilde{u}}_r,\V{\tilde{v}}_1,\ldots,\V{\tilde{v}}_{m-r}$  
to obtain the orthogonal matrices $W$ and $\Wt$, whose columns constitute 
a basis pair similar to the one above. In this basis pair, 
\begin{eqnarray}
  W'\,X\,W &=&  \Big[diag(x_1,\ldots,x_r) \oplus I_r\Big] \oplus I_{m-2r}\\
  W'\,\hat{X}_\eta(Y) \,W &=&  R_n \oplus I_{n-2r}
\end{eqnarray}
where $R_n$ is a sequence of $2r$-by-$2r$ matrices such that 
\[
  R_n \aslim \oplus_{i=1}^r B\big(\eta(y_i),c_1,s_1,\tilde{c}_1,\tilde{s}_1\big)\,.
\]
The lemma follows by permuting the coordinates, and then using the invariance
and the decomposability
properties of the loss family $L$.

\subsection{Deterministic formula for the asymptotic loss}

In Section \ref{block:subsec} we analyzed a single matrix and shown that, for
fixed $m$ and $n$, the
loss $L_{m,n}(X\,,\,\hat{X}_\eta(Y))$ decomposes  to ``atomic'' units of the
form 
\[L_{2,2}\Big(A(x_i)\,,\,B(\eta(y_i),c_i,s_i,\ct_i,\stt_i)\Big)\,.\]

Let us now
return to the sequence model $Y_n=X_n+Z_n/\sqrt{n}$ and find the limiting value
of these ``atomic'' units as $n\to\infty$. This will lead to a simple formula
for the asymptotic loss $L_\infty(\eta|\V{x})$.

Each of these ``atomic'' units only depend on $y_i$, the $i$-th data singular value, and
on $c_i$ (resp. $\ct_i$), the angle between the $i$-th left (resp. right) signal
and data singular vectors. In the special case of Frobenius norm loss, we
have already encountered this phenomenon (Lemma \ref{MSE:lem}), where we have
seen that these quantities converge to deterministic functions that depend on $x_i$, the
$i$-th  signal singular value alone.

For the sequence $Y_n=X_n+Z_n/\sqrt{n}$, recall our notation 
$y_{n,i},\V{u}_{n,i},\V{\ut}_{n,i},\V{v}_{n,i},\V{\vt}_{n,i}$ 
from \eqref{Xn:eq} and \eqref{Yn:eq}, and define
\begin{eqnarray*}
  c_{n,i}=\inner{\V{u}_{n,i}}{\V{v}_{n,i}} && s_i=\sqrt{1-c_{n,i}^2} \\
  \ct_{n,i}=\inner{\V{\ut}_{n,i}}{\V{\vt}_{n,i}} && \stt_i=\sqrt{1-\ct_{n,i}^2} \,,
\end{eqnarray*}
for $i=1,\ldots,r$. Combining Lemma \ref{blockmany:lem}, Lemma \ref{y-asy:lem} and Lemma \ref{inner-asy:lem} we obtain:

\begin{lemma} \label{block_conv:lem}
  Let $Y_n=X_n+Z_n/\sqrt{n}$ be a matrix sequence in our asymptotic framework
  with signal singular values $\V{x}=(x_1,\ldots,x_r)$. 
  Assume that  $\eta$ is continuous at 
  $y(x_i)$ for some fixed $1\leq i\leq r$. If $\beta^{1/4}\leq x_i$ then
  \[  
  \lim_{n\to\infty} L_{2,2}\left(
  A(x_i)\,,\,B(\eta(y_{n,i}),c_{n,i},s_{n,i},\ct_{n,i},\stt_{n,i}
  \right) \aseq
  L_{2,2}\Big(A(x_i)\,,\,B(\eta(y(x_i)),x_i)\Big)\,,
\]
where $B(\eta,x)$ is given by \eqref{B:eq}, while if $0\leq x_i < \beta^{1/4}$ then
\begin{eqnarray*} 
  \lim_{n\to\infty} L_{2,2}\left(
  A(x_i)\,,\,B(\eta(y_{n,i}),c_{n,i},s_{n,i},\ct_{n,i},\stt_{n,i}
  \right) &\aseq&
  L_{2,2}\Big(A(x_i)\,,\,diag(0,\eta(y(x_i)))\Big)\\
  &=& L_{1,1}(x_i,0) + L_{1,1}(0,\eta(y(x_i)))\,.
\end{eqnarray*}
\end{lemma}

As a result, we now obtain the
asymptotic loss $L_\infty$ as a deterministic function of the nonzero 
signal singular values $x_1,\ldots,x_r$.
Observe that by Lemma \ref{fixedrank:lemma}, if $\eta$ is a conservative shrinker,
then eventually $\eta(y_{n,i})=0$ for all $i>r$. Therefore the assumption
$\eta(y_i)=0$ for $i>r$, required for Lemma \ref{blockmany:lem},  is satisfied
eventually. Combining Lemma \ref{blockmany:lem} and Lemma 
\ref{block_conv:lem}, we obtain 

\begin{lemma} \label{char2:lem}
  {\bf A formula for the asymptotic loss of a conservative shrinker.}
  Assume that $L=\left\{ L_{m,n} \right\}$ is a sum- or max- decomposable family
of orthogonally invariant losses.
Extend the definition of $B(\eta,x)$ from \eqref{B:eq} by setting 
$B(\eta,x)=diag(0,\eta)$ for $0\leq x<\beta^{1/4}$. 
If $\eta:[0,\infty)\to[0,\infty)$ is a conservative shrinker,
  then
  \begin{eqnarray} \label{L_infty_sum:eq}
    L_\infty(\eta|\V{x}) = \sum_{i=1}^r
    L_{2,2}\Big(A(x_i),B(\eta(y(x_i)),x_i)\Big)
  \end{eqnarray}
if $L$ is sum-decomposable, or
    \begin{eqnarray} \label{L_infty_max:eq}
    L_\infty(\eta|\V{x}) = \max_{i=1}^r
    L_{2,2}\Big(A(x_i),B(\eta(y(x_i)),x_i)\Big)
  \end{eqnarray}
  if $L$ is max-decomposable.
\end{lemma}

The final step toward the proof of Theorem \ref{char:thm} involves the case when
the shrinker $\eta$ is given as a special concatenation of two conservative shrinkers.
Even if the two parts of $\eta$ do not match, forming a discontinuity point in
which the limits from the left and from the right disagree, we may
still have a formula for the asymptotic loss -- provided that
the loss functions match.

\begin{defn}
  Assume that there exist a point $0< x^*$ and two shrinkers, 
$\eta_1:[0,x^*)\to[0,\infty)$ and $\eta_2:[x^*,\infty)\to[0,\infty)$, 
 such that 
  \[
    L_{2,2}\Big(A(x^*),B(\eta_1(y(x^*)),x^*)\Big) = 
    L_{2,2}\Big(A(x^*),B(\eta_2(y(x^*)),x^*)\Big)\,.
  \]
  We say that the asymptotic loss functions of $\eta_1$ and $\eta_2$ {\em cross}
  at $x^*$.
\end{defn}

\begin{lemma} \label{char3:lem}
  {\bf A formula for the asymptotic loss of a concatenation of two conservative
  shrinkers.}
  Assume that $L=\left\{ L_{m,n} \right\}$ is a sum- or max- decomposable family
of orthogonally invariant losses.
Extend the definition of $B(\eta,x)$ from \eqref{B:eq} by setting 
$B(\eta,x)=diag(0,\eta)$ for $0\leq x<\beta^{1/4}$. 
Assume that there exist  two shrinkers, 
$\eta_1:[0,x^*)\to[0,\infty)$ and $\eta_2:[x^*,\infty)\to[0,\infty)$, 
  whose asymptotic loss functions cross at some point $0<x^*$.
  Define 
\[
\eta(x)=
\begin{cases}
  \eta_1(x) & 0\leq x \leq x^* \\
  \eta_2(x) & x^* < x
\end{cases}\,.
\]
  Then $L_\infty(\eta|\cdot)$ exists and is given by 
  \eqref{L_infty_sum:eq} if $L$ is sum-decomposable, or \eqref{L_infty_max:eq}
  if $L$ is max-decomposable.
\end{lemma}

\paragraph{Proof of Theorem \ref{char:thm}.}

Consider the shrinker $\eta_1 \equiv 0$. By Lemma \ref{block_conv:lem}, $\eta_1$
dominates any other conservative shrinker when
$0\leq x<\beta^{1/4}$. By assumption, there exists a point $\beta^{1/4}\leq x_0$
such that $\eta_1$ also dominates any conservative shrinker on
$[\beta^{1/4},x_0)$, and such that $\eta^{**}$ dominates any other conservative
  shrinker on $[x_0,\infty)$. Finally, by assumption, the asymptotic loss
    functions of $\eta_1$ and $\eta^{**}$ cross at $x_0$. By Lemma
    \ref{char3:lem}, the concatenated shrinker $\eta^*$ dominates any conservative
    shrinker on $[0,\infty)$.
\qed

\section{Finding Optimal Shrinkers Analytically: Frobenius, Operator \& Nuclear Losses} 
\label{froopnuc:sec}

Theorem \ref{char:thm} provides a general recipe for finding optimal singular
value shrinkers, which was provided
in Section \ref{discuss:subsec}. To see it in action, we turn to our three primary examples,
namely, the Frobenius norm loss, the operator norm loss and the nuclear norm
loss. In this section we find explicit formulas for the optimal singular value
shrinkers in each of these losses.

We will need the following lemmas  regarding $2$-by-$2$ matrices (see \cite{Donohoa}):

\begin{lemma} \label{M:lem}
  The eigenvalues of any  $2$-by-$2$ matrix $M$ with trace $trace(M)$ and determinant
  $det(M)$ are given by
  \begin{eqnarray}
    \lambda_\pm(M) = \tfrac{1}{2}\left( trace(M)\pm\sqrt{trace(M)^2-4det(M)}
    \right)\,.
  \end{eqnarray}
\end{lemma}
 
\begin{proof}
  These are the roots of the characteristic polynomial of $M$.
\end{proof}

\begin{lemma} \label{sigma_dot:lem}
  Let $\Delta$ be a $2$-by-$2$ matrix with singular values $\sigma_+>\sigma_->0$. Define $t=trace(\Delta\Delta')=\norm{\Delta}_F^2$, $d=det(\Delta)$ and
$r^2=t^2-4d^2$. Assume that $\Delta$ depends on a parameter $\eta$ and let
$\dot{\sigma}_\pm$, $\dot{t}$ and $\dot{d}$ denote the derivative of these
quantities w.r.t the parameter $\eta$. Then
\[
  r(\dot{\sigma}_++\dot{\sigma}_-)(\dot{\sigma}_+-\dot{\sigma}_-) = 
  2(\dot{t}+2\dot{d})(\dot{t}-2\dot{d})\,.
\]
\end{lemma}
\begin{proof}

  By Lemma \ref{M:lem} we have 
  $
    2\sigma_\pm^2=t\pm r
  $ and therefore
  \[
    \sqrt{2}\dot{\sigma}_\pm = \frac{\dot{t}\pm\dot{r}}{2\sqrt{t\pm r}}\,.
  \]
  Differentiating and expanding $\dot{\sigma}_+ \pm \dot{\sigma}_-$ we obtain the
  relation
  \begin{eqnarray}    \label{sigmapm_dot:eq}
    \left( \dot{\sigma}_+ + \dot{\sigma}_- \right) = 
    \frac{(8d/r)
      (\dot{t}+2\dot{d})
      (\dot{t}-2\dot{d})}{(t^2-r^2)(\dot{\sigma}_+-\dot{\sigma}_-)}
    \end{eqnarray}
    and the result follows.
\end{proof}

\begin{lemma} \label{Delta:lem}
  Let $\eta,c,\ct\geq0$ and set $s=\sqrt{1-c^2}$ and $\stt=\sqrt{1-\ct^2}$.
Define
  \[
  \Delta = \Delta(\eta,c,\ct,s,\stt) = \begin{bmatrix}  \eta c\ct -x  & \eta c\stt  \\
  \eta \ct s & \eta s\stt \,.
\end{bmatrix}
\]
Then
\begin{eqnarray}
  \norm{\Delta}_F^2 &=& \eta^2 + x^2 -2x\eta c \ct \label{tr:eq}\\
  det(\Delta) &=&  -x\eta s\stt\,, \label{det:eq}
\end{eqnarray}
  and the singular values $\sigma_+>\sigma_-$ of $\Delta$ are given by 
  \begin{eqnarray}
    \sigma_\pm = \tfrac{1}{\sqrt{2}}\sqrt{\norm{\Delta}_F^2 \pm \sqrt{\norm{\Delta}_F^4 -
  4det(\Delta)^2}}\,. \label{sigmapm:eq}
  \end{eqnarray}
\end{lemma}
\begin{proof}
  Apply Lemma \ref{M:lem} to the matrix $\Delta$.
\end{proof}

\subsection{Frobenius norm loss}

Theorem \ref{char:thm} allows us to rediscover the optimal shrinker for Frobenius
norm loss, which was derived from first principles in Section \ref{fro:sec}.
To this end, observe that  by \eqref{tr:eq} we have
\begin{eqnarray} \label{L22_fro:eq}
  L^{fro}_{2,2}\left( 
   \eta \begin{bmatrix}  c\ct  & c\stt  \\
  \ct s & s\stt 
\end{bmatrix} \,,\, 
\begin{bmatrix}  x  & 0  \\
  0 & 0
\end{bmatrix} 
  \right) = 
  \norm{\Delta}_F^2 
= \eta^2 + x^2 -2x\eta c \ct\,.
\end{eqnarray}
To find the optimal shrinker, we solve
$\partial\norm{\Delta}_F^2/\partial\eta=0$ for $\eta$ and use the fact that
$c^2\ct^2+c^2\stt^2+s^2\ct^2+s^2\stt^2=1$. We find that the minimizer of
$\norm{\Delta}_F^2$ is
$\eta^{**}(x) = x\,c\ct$. 
Defining $\eta^{**}(x)=x\,c(x)\ct(x)$ for $x\geq
\beta^{1/4}$,
we find that the asymptotic loss of $\eta^{**}(x)$ and of $\eta\equiv 0$
cross at $x_0=\beta^{1/4}$. 
Simplifying \eqref{concat:eq},
where $x(y)$ is given by \eqref{x_of_y:eq}, 
we find that $\eta^{*}(y)$ is given by \eqref{opt_shrinker_fro:eq}. 
By Theorem \ref{char:thm}, this is the
optimal shrinker.  (As mentioned in  Section \ref{fro:sec}, the optimal shrinker
is not itself strictly conservative, yet can be shown to have the asymptotic loss
predicted by our framework.)

\subsection{Operator norm loss}

By \eqref{sigmapm:eq},
\begin{eqnarray}
  L^{op}_{2,2}\left( 
   \eta \begin{bmatrix}  c\ct  & c\stt  \\
  \ct s & s\stt 
\end{bmatrix} \,,\, 
\begin{bmatrix}  x  & 0  \\
  0 & 0
\end{bmatrix} 
  \right) = 
  \norm{\Delta}_{op} = \sigma_+\,.
\end{eqnarray}
To find the optimal shrinker, we solve
$\partial\norm{\Delta}_{op}/\partial\eta=0$ for $\eta$ on $\beta^{1/4}\leq x$.
 We find that 
 the minimizer of $\norm{\Delta}_{op}$ is at $\eta^{**}(x)=x$. 
 By \eqref{sigmapm:eq}, the asymptotic loss of $\eta^{**}$ is given by
 $x\sqrt{1-c(x)\ct(x)+|c(x)-\ct(x)|}$, so that the asymptotic loss of
 $\eta^{**}(x)$ and of $\eta\equiv 0$ cross at $x_0=\beta^{1/4}$.  
 Simplifying \eqref{concat:eq}, 
we recover the shrinker $\eta^*(y)$ mentioned above in \eqref{opt_shrinker_op:eq}. 
Observe that although this shrinker is discontinuous at
$y(\beta^{1/4})=1+\sqrt{\beta}$, its asymptotic loss $L_\infty(\eta^*|\cdot)$
exists by Lemma \ref{char3:lem}, and, by Theorem
\ref{char:thm}, dominates the asymptotic loss any conservative shrinker. We note that
as in the Frobenius case, this optimal shrinker is not strictly conservative, as it
only satisfies $\eta(y)=0$ for $y\leq \beta_+$. Here too a delicate argument
beyond our present scope shows that the asymptotic loss exists and matches the
formula predicted by our framework \cite{Donohoa}.

\paragraph{Remark.} The optimal shrinker for operator norm loss $\eta^*(y) =
x(y)$ simply shrinks the data singular value back to the "original" location of
its corresponding signal singular value.

\subsection{Nuclear norm loss}

Again by \eqref{sigmapm:eq}
\begin{eqnarray}
  L^{nuc}_{2,2}\left( 
   \eta \begin{bmatrix}  c\ct  & c\stt  \\
  \ct s & s\stt 
\end{bmatrix} \,,\, 
\begin{bmatrix}  x  & 0  \\
  0 & 0
\end{bmatrix} 
  \right) = 
  \norm{\Delta}_* = \sigma_++\sigma_-\,. 
\end{eqnarray}

\noindent To find the optimal shrinker, assume first that $x$ is such that $c(x)\ct(x)\geq
s(x)\stt(x)$. 
By Lemma
\ref{sigma_dot:lem}, with
\[
  t = \eta^2+x^2-2x\eta c \ct \qquad \text{and} \qquad d = -x\eta s \stt\,,
\]
we find that only zero of $\partial(\sigma_++\sigma_-)/\partial \eta$ occurs when 
$\partial t/\partial \eta - \partial d/\partial \eta=0$, namely at
\[
  \eta^{**}(x) = x\Big(c(x)\ct(x) - s(x)\stt(x)\Big)\,.
\]
Direct calculation using \eqref{sigmapm:eq}, \eqref{det:eq} and \eqref{tr:eq} 
shows that the square of the asymptotic loss of $\eta^{**}$ is simply 
$x^2+(\eta^{**}(x))^2 -2x\eta^{**}(x)\Big(c(x)\ct(x)-s(x)\stt(x)\Big)$,
so that the asymptotic loss of $\eta^{**}(x)$ and of $\eta\equiv 0$ cross at 
the unique $x_0$ satisfying $c(x_0)\ct(x_0)=s(x_0)\stt(x_0)$.
Substituting $c=c(x)$ from \eqref{c:eq}, $\ct=\ct(x)$ from \eqref{ct:eq} and
also $s=s(x)=\sqrt{1-c(x)^2}$ and $\stt=\stt(x)=\sqrt{1-\ct(x)^2}$, we find 
\[
  \eta^*(y) = \left(\frac{x^4-\beta}{x^2y} + \frac{\sqrt{\beta}}{x}\right)_+\,,
\]
recovering the optimal shrinker \eqref{opt_shrinker_op:eq}. Inspection of 
\eqref{opt_shrinker_op:eq} reveals that this
optimal shrinker is in fact a conservative shrinker.

\section{Finding Optimal Shrinkers Numerically: Schatten norm losses}
\label{schatten:sec}

In Section \ref{froopnuc:sec} we have followed the recipe discussed in Section
\ref{discuss:subsec} analytically, and explicitly solved for the optimal shrinkers
of the Frobenius, Operator and Nuclear
norm losses. 
In some cases, the optimization problem \eqref{Fmin:eq} does not admit 
a closed-form solution, and in other cases, the closed-form solution is 
unreasonably
complicated. For such cases, we note that it is extremely easy to solve the problem  \eqref{Fmin:eq}
numerically, as it only involves minimization of a univariate function that depends on
the two eigenvalues of a $2$-by-$2$ matrix. 
To demonstrate that our recipe for finding optimal shrinkers can
be easily executed numerically, rather than analytically, 
in this section we find the optimal shrinker for
any Schatten-$p$ norm loss numerically\footnote{We thank the anonymous referee for
 this helpful suggestion.}, for any value $p>0$. 

Let $0< p \leq \infty$. Recall that the Schatten-$p$ norm (resp. quasi-norm if
$0<p<1$) of a matrix is the
$\ell_p$ norm (resp. quasi-norm) of its singular values vector. If the singular values of the
$m$-by-$n$
matrix $X-\hat{X}$ are $\sigma_1,\ldots,\sigma_m$, define the Schatten-$p$ 
loss 
\begin{eqnarray*}
  L_{m,n}^{S_p}(X,\hat{X}) &=&  \sqrt[p]{\sum_{i=1}^m \sigma_i^p} \qquad\qquad
  0<p<\infty\\
  L_{m,n}^{S_p}(X,\hat{X}) &=& \max\left\{ \sigma_1,\ldots,\sigma_m \right\} \qquad p=\infty\,,
\end{eqnarray*}
where the matrix size $m,n$ has been suppressed in the notation for simplicity.

Schatten-$p$ norms and quasi-norms have been considered in the literature for matrix estimation:
see \cite{Rohde2011,Koltchinskii2011a} and references therein. (The case $0<p\leq
1$  is of special interest in matrix completion problems due to its low-rank inducing
behavior.)
So far in this paper we have carefully studied three special cases: $L_{m,n}^{fro}\equiv
L_{m,n}^{S_2}$,
 $L_{m,n}^{nuc}\equiv L_{m,n}^{S_1}$ and
 $L_{m,n}^{op}\equiv L_{m,n}^{S_\infty}$.
 Observe that for any $0< p < \infty$, the Schatten-$p$ loss is
 orthogonally invariant and sum-decomposable, hence amenable to the our
 analysis.

 While it is in principle possible to derive the optimal shrinker for the
 Schatten-$p$ loss analytically using Lemma \ref{sigma_dot:lem} and Lemma
\ref{Delta:lem}, the result would be a very complicated expression. Instead, we
follow the recipe of Section \ref{discuss:subsec} numerically:
We select points of interest $\left\{ y_i \right\}$ in which we would like to
evaluate the optimal shrinker $\eta^*(y)$. 
We define  $x_i=x(y_i)$ where $y\mapsto x(y)$ is the transformation from
    \eqref{x_of_y:eq}.
    For each of the values $\{x_i\}$ we form a symbolic expression for the function 
 $F(\eta,x_i)$ from \eqref{F:eq}, and minimize it numerically to obtain 
 the minimizer $\eta^{**}(x_i)$ from \eqref{Fmin:eq}.
 The desired value of the optimal shrinker $\eta^*(y_i)$ is then given by
 $\eta^*(y_i) = \eta^{**}(x(y_i)) = \eta^{**}(x_i)$.

 Figure \ref{schatten1:fig} and Figure \ref{schatten2:fig}
 show the optimal shrinker discovered numerically for
 the Schatten-$p$ loss, for a few values of $p$. 
 Figure \ref{schatten1:fig} focuses on the case $p\geq 1$, where the
 Schatten-$p$ loss is given by a norm. 
 Note the familiar shapes for
 the values $p=1,2,10000$ (the latter is indistinguishable from the case
 $p=\infty$, namely the operator norm). It seems that the optimal shrinker for
 all cases $1\leq p<\infty$ are continuous, and that the discontinuity found
 analytically for the case $p=\infty$ forms only in the limit $p\to\infty$. 
 Figure \ref{schatten2:fig} focuses on the case $0<p<1$, where the Schatten-$p$
 loss is given by a quasi-norm. The numerical findings are fascinating and
 prompt further research: for instance, while the optimal shrinker for $p=1$ is
 continuous, at an unknown value $p<1$ the shrinkers become discontinuous,
 with a discontinuity resembling that of the $p=\infty$ case. Furthermore, for
 small values of $p$, the
 optimal shrinkers are very similar to the hard thresholding nonlinearities,
 with a ``hard threshold'' that depends on $p$ and on the aspect ratio $\beta$.
 In other words, in these cases, the optimal shrinker and the optimal hard
 thresholding
 nonlinearity seem to approximately coincide. It also seems that as $p\to 0$, the optimal
 shrinkers tend to the zero shrinker. All these phenomena can be studied and
 evaluated precisely in further research using the framework developed in this paper.

\begin{figure}[h!]
\centerline{
  \includegraphics[height=2.2in]{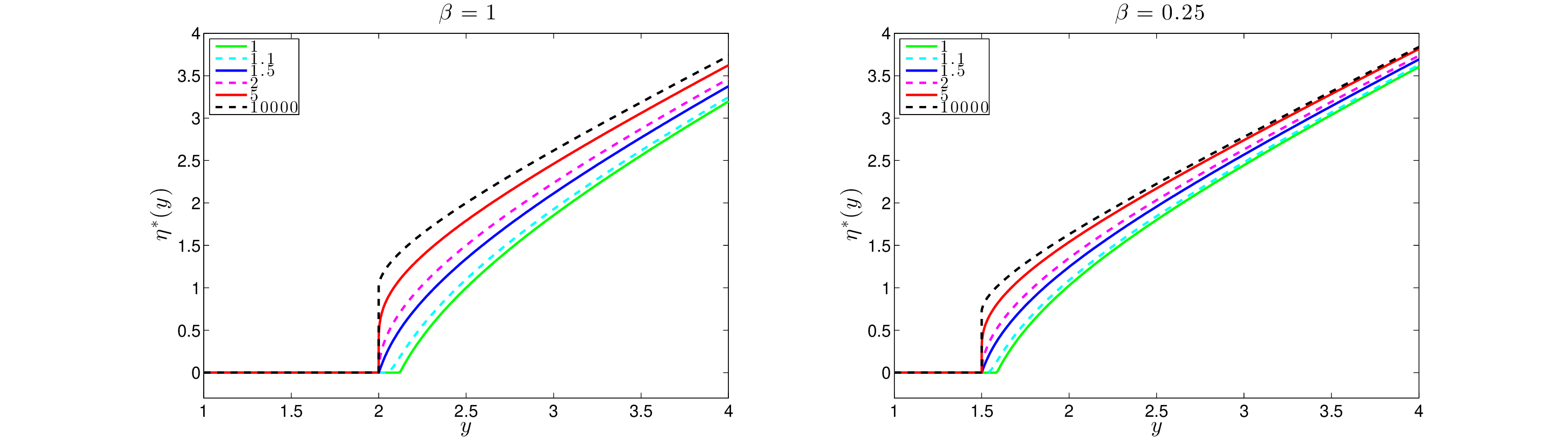}}
   \caption{\small 
     Numerically computed optimal shrinkers for the Schatten-$p$ loss, for
     different values of $p\geq 1$ shown in the legend. Left, $\beta=1$ (the case of
     square matrix); Right, $\beta=0.25$ (the case of four times as many columns
   as rows.) Each shrinker was evaluated on a grid of $200$ points.
 This figure can be reproduced
 using the code supplement \cite{Gavish2015}. }
   \label{schatten1:fig}
 \end{figure}

 \begin{figure}[h!]
\centerline{
  \includegraphics[height=2.2in]{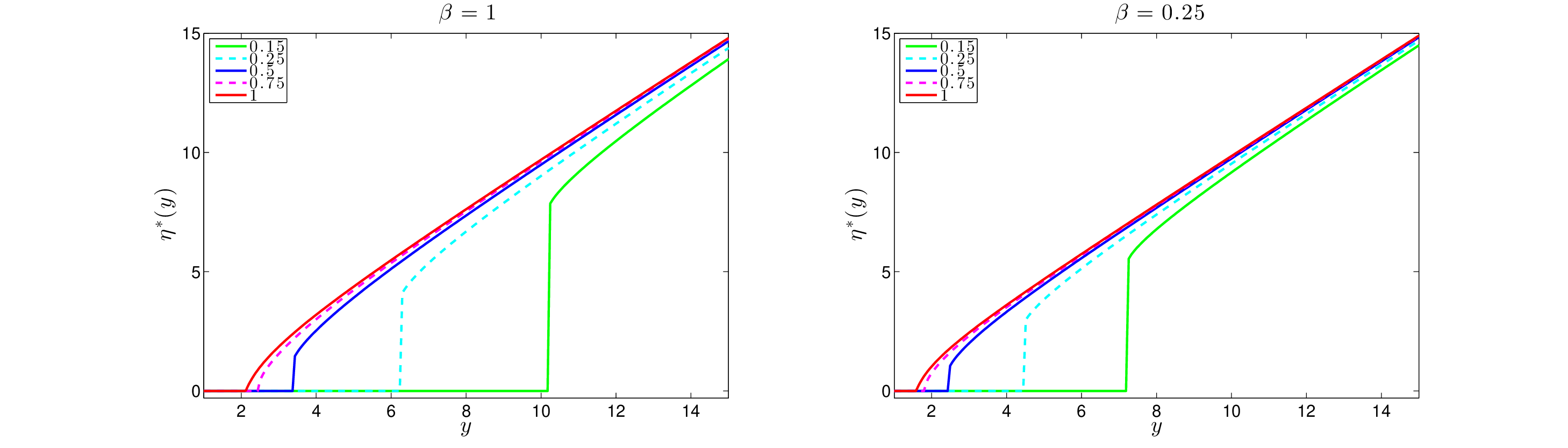}}
   \caption{\small 
     Numerically computed optimal shrinkers for the Schatten-$p$ loss, for
     different values of $0<p\leq 1$ shown in the legend. Left, $\beta=1$ (the case of
     square matrix); Right, $\beta=0.25$ (the case of four times as many columns
   as rows.) Each shrinker was evaluated on a grid of $200$ points.
 This figure can be reproduced
 using the code supplement \cite{Gavish2015}. }
   \label{schatten2:fig}
 \end{figure}

\section{Extensions}

Our main results have been formulated and calibrated specifically for
the model $Y=X+Z/\sqrt{n}\in\Mmn$, where the distribution of the noise matrix
$Z$ is orthogonally invariant.
In this section we extend our main results to include the model $Y=X+\sigma
Z\in\Mmn$, and consider:
\begin{enumerate}

  \item The setting where $\sigma$ is either known but does not necessarily
    equal $1/\sqrt{n}$, or is altogether unknown.

  \item The setting where the noise matrix $Z$ has i.i.d entries, but its
    distribution is not necessarily orthogonally invariant.

 \end{enumerate}

The results below follow \cite{Donoho2013b}.

\subsection{Unknown Noise level} \label{noise:sec}

Consider an asymptotic framework slightly more general than the one in Section
\ref{framework:subsec}, in which $Y_n = X_n + (\sigma/\sqrt{n})Z_n$, with $X_n$
and $Z_n$ as defined there. In this section we keep the loss family $L$ and the
asymptotic aspect ratio $\beta$ fixed and implicit. We extend Definition \ref{asyloss:def}
 and write
 \[
     L_\infty(\eta|\V{x},\sigma) \aseq \lim_{n\to\infty} L_{m_n,n}\left(X_n\,,\,\hat{X}_\eta(X_n +
     \tfrac{\sigma}{\sqrt{n}}Z_n)\right)\,.
   \]

When the noise level $\sigma$ is known, Eq.
\eqref{scale-xhat:eq} allows us to re-calibrate any nonlinearity
$\eta$,
originally calibrated for noise level $1/\sqrt{n}$, to a different noise level.
For a nonlinearity $\eta:[0,\infty)\to[0,\infty)$, write
  \[
    \eta_c(y) = c\cdot \eta(y/c)\,.
  \]
  We clearly have: 

  \begin{lemma}
If $\eta^*$ is an optimal shrinker for $Y_n=X_n+Z_n/\sqrt{n}$,
  namely,
\[
  L_\infty(\eta^*|\V{x}) \leq L_\infty(\eta|\V{x})
\]
for any $r\geq 1$, any $\V{x}\in\R^r$ and any continuous nonlinearity $\eta$, and $\sigma>0$, then $\eta^*_\sigma$ is an
optimal shrinker for $Y_n=X_n+(\sigma/\sqrt{n})Z_n$, namely
\[
  L_\infty(\eta^*_\sigma|\V{x},\sigma) \leq L_\infty(\eta|\V{x},\sigma)
\]
for any $r\geq 1$, any $\V{x}\in\R^r$ and any continuous nonlinearity $\eta$.
\end{lemma}

When the noise level $\sigma$ is unknown, we are required to estimate it. See
\cite{Kritchman2009,Owen2009,Perry2009,Shabalin2013,Nadakuditi,Josse2016} 
and references therein for existing literature
on this estimation problem. The method below has been proposed in 
\cite{Donoho2013b}.

Consider the following robust estimator for the
parameter $\sigma$ in the model $Y=X+\sigma Z$:
\begin{eqnarray}
  \hat{\sigma}(Y) = \frac{y_{med}}{\sqrt{n \cdot\mu_\beta}}\,,
\end{eqnarray}
where $y_{med}$ is a median singular value of $Y$ and $\mu_\beta$ is the
median of the Marcenko-Pastur distribution, namely, the unique solution 
in $\beta_-\leq x\leq \beta_+$
to the equation
\begin{eqnarray*}
  \intop_{\beta_-}^x \frac{\sqrt{(\beta_+-t)(t-\beta_-)}}{2\pi t}dt =
  \frac{1}{2}\,,
\end{eqnarray*}
where $\beta_\pm = 1\pm\sqrt{\beta}$.
Note that the median $\mu_\beta$ is not available analytically but can easily be obtained
by numerical quadrature.

\begin{lemma} \label{sigmahat:lem}
  Let $\sigma>0$.
  For the sequence $Y_n=X_n+(\sigma/\sqrt{n})Z_n$ in our asymptotic framework, 
  \[\lim_{n\to\infty} \frac{\hat{\sigma}(Y_n)}{1/\sqrt{n}}\aseq  \sigma .\]
\end{lemma}
\begin{cor}
Let $Y_n=X_n+(\sigma/\sqrt{n})Z_n$ be a sequence in our asymptotic framework and
let $\eta^*$ be an optimal shrinker calibrated for $Y_n=X_n+Z_n/\sqrt{n}$.
Then the random sequence of shrinkers $\eta^*_{\hat{\sigma}(Y_n)}$ 
converges to the optimal shrinker $\eta^*_\sigma$:
\[ 
  \lim_{n\to\infty} \eta_{\hat{\sigma}(Y_n)}(y) \aseq \eta_\sigma(y)\,,\qquad
  y>0 \,.
\]
Consequently,  $\eta^*_{\hat{\sigma}(Y_n)}$ asymptotically achieves optimal
performance:
\[
  \lim_{n\to\infty} L_{m_n,n}\left(X_n\,,\,\hat{X}_{\eta^*_{\hat{\sigma}(Y_n)}}(X_n +
     \tfrac{\sigma}{\sqrt{n}}Z_n)\right) = 
   L_\infty(\eta^*_\sigma|\V{x},\sigma)\,.
\]
\end{cor}

In practice, for denoising a matrix $Y\in\Mmn$, assumed to satisfy $Y=X+\sigma
Z$, where $X$ is low-rank and $Z$ has i.i.d entries, we have the following 
approximately optimal
singular value shrinkage estimator:
\begin{eqnarray}
  \hat{X}(Y) = \sqrt{n}\sigma \hat{X}_{\eta^*}(Y/(\sqrt{n}\sigma))
\end{eqnarray}
when $\sigma$ is known, and 
\begin{eqnarray}
  \hat{X}(Y) = \sqrt{n}\hat{\sigma}(Y)\cdot
  \hat{X}_{\eta^*}(Y/(\sqrt{n}\hat{\sigma}(Y)))
\end{eqnarray}
when $\sigma$ is unknown.
Here, $\eta^*$ is an optimal shrinker with respect to
desired loss family $L$ in the natural scaling.

\subsection{General white noise} \label{general_noise:sec}

Our results were formally stated for the sequence of models of the form
$Y=X+\sigma Z$, where $X$ is a non-random matrix to be estimated, and the entries
of $Z$ are i.i.d samples from a distribution that is orthogonally invariant (in
the sense that the matrix $Z$ follows the same distribution as $A Z B$, for
any orthogonal $A\in M_{m,m}$ and $B\in M_{n,n}$). While Gaussian noise is orthogonally
invariant, many common distributions, which one could consider to model white
observation noise, are not. 

The singular values of a signal matrix $X$ constitute a very widely used measure of the
complexity, or information content, of $X$. In particular, they capture its rank. One attractive feature of the framework we adopt 
is that the loss $L_{m,n}(X,\hat{X})$ only depends
on the signal matrix $X$ through its 
nonzero singular values $\V{x}$. This allows the loss to be directly related to
the complexity of the signal $X$.
If the
distribution of $Z$ is not orthogonally invariant, the loss no longer enjoys  this
 property. This point is
 discussed extensively in  \cite{Shabalin2013}.

In general white noise, which is not necessarily orthogonally invariant, one can
still allow the loss to depend on $X$ only through its singular values by 
placing a prior distribution on $X$ and shifting to a model where it is a
random, instead of a fixed, matrix.  Specifically, 
consider an alternative asymptotic framework to the one in Section
\ref{framework:subsec}, in which the sequence denoising problems
$Y_n=X_n+Z_n/\sqrt{n}$ satisfies the following assumptions:
\begin{enumerate}
  \item {\em General white noise:} The entries of $Z_n$ are i.i.d
    samples from a
    distribution with zero mean, unit variance and finite fourth moment.  
\item {\em Fixed signal column span and uniformly distributed signal singular vectors:}
Let the rank $\rk>0$ be fixed and choose a vector $\V{x}\in\R^\rk$ with
coordinates $\V{x}=(x_1,\ldots,x_\rk)$. Assume that for all $n$, 
\begin{eqnarray} \label{singvec2:eq}
X_n = U_n \, diag(x_1,\ldots,x_r,0,\ldots,0) \, V_n'\,
\end{eqnarray}
is a singular value decomposition of $X_n$, where $U_n$ and $V_n$ are 
uniformly distributed random orthogonal matrices. Formally, $U_n$ and $V_n$ are
sampled from the Haar distribution on the
$m$-by-$m$ and $n$-by-$n$ orthogonal group, respectively.

  \item {\em Asymptotic aspect ratio $\beta$:}
    The sequence $\m_\n$ is such that $\m_\n / \n \to \beta$. 
 \end{enumerate}

 The second assumption above implies that $X_n$ is a ``generic'' choice of matrix
 with nonzero singular values $\V{x}$, or equivalently, a generic choice of
 coordinate systems in which the linear operator corresponding to $X$ is
 expressed.

The results of \cite{Benaych-Georges2012}, which we have used, hold in this case
as well. It follows that Lemma \ref{y-asy:lem} and Lemma \ref{inner-asy:lem}, and
consequently all our main results, hold under this alternative framework. In
short, in general white noise, all our results hold if one is willing to only
specify the signal singular values, rather than the signal matrix, and consider
a ``generic'' signal matrix with these singular values.

\section{Simulation}

Our results are exact only in the limit as the matrix size grows
to infinity. 
To study the accuracy of the asymptotic loss on finite matrices, and to compare
the optimal shrinker with optimally tuned hard and soft thresholding, we conducted
two simulation studies. 

\paragraph{Comparing the asymptotic loss with the empirical loss.} We studied $n$-by-$n$ matrices of the form $Y=X+Z$. The
signal matrix had exactly $r$ identical nonzero singular values. For brevity, we
focused on the asymptotic Frobenius loss.
Figure \ref{fig:fig7_a} compares the case $(n,r)=(20,1)$ with the case $(n,r)=(100,1)$. 
Figure \ref{fig:fig7_b} compares the case $(n,r)=(50,2)$ with the case $(n,r)=(50,4)$.
In each case we show three different noise distributions: the entries of the 
noise matrix $Z$ are i.i.d draws from a Gaussian distribution (thin tails),
uniform distribution (no tails) and Student-t with 6 degrees of freedom (fat
tails). We overlay the predicted asymptotic loss from Eq. \eqref{L_infty_sum:eq} 
and the observed loss for
different values of the signal singular value $x$. The observed loss was obtained by
averaging 50 Monte Carlo iterations. The shrinkers shown are the optimal
shrinker for Frobenius loss from Eq. \eqref{opt_shrinker_fro:eq}, and the
optimally tuned hard and soft thresholds as described in Section \ref{hard_and_soft:subsec}. 
Simulations show qualitatively that our results are useful already for
relatively small matrices, and that the low-rank assumption remains valid when
$r/m\leq0.1$, say.

\paragraph{Comparing optimal shrinkers with a brute-force calculation of the
optimal shrinkage.}

We studied $20$-by-$20$ matrices of the form $Y=X+Z$. The signal matrix was
rank-$1$ and the noise matrix was i.i.d Gaussian. For each of the three losses 
$\{$ Frobenius, nuclear, operator $\}$, we calculated the optimal
shrinkers using brute-force by scanning over a grid of possible values $\eta$
and finding the value that minimized the empirical loss as calculated by
averaging over $10$ monte carlo draws. Figure \ref{empirical_opt:fig} overlays 
the shrinkage calculated by brute-force over the asymptotically optimal
shrinkers calculated for the three losses in Section \ref{froopnuc:sec}.
Note the agreement with the asymptotic formulae already for $n=20$ and rank
fraction of $1/20=0.05$.

\begin{figure}[h]
  \centering
  \includegraphics[width=6.5in, trim= 140 1 3 .5, clip=true]{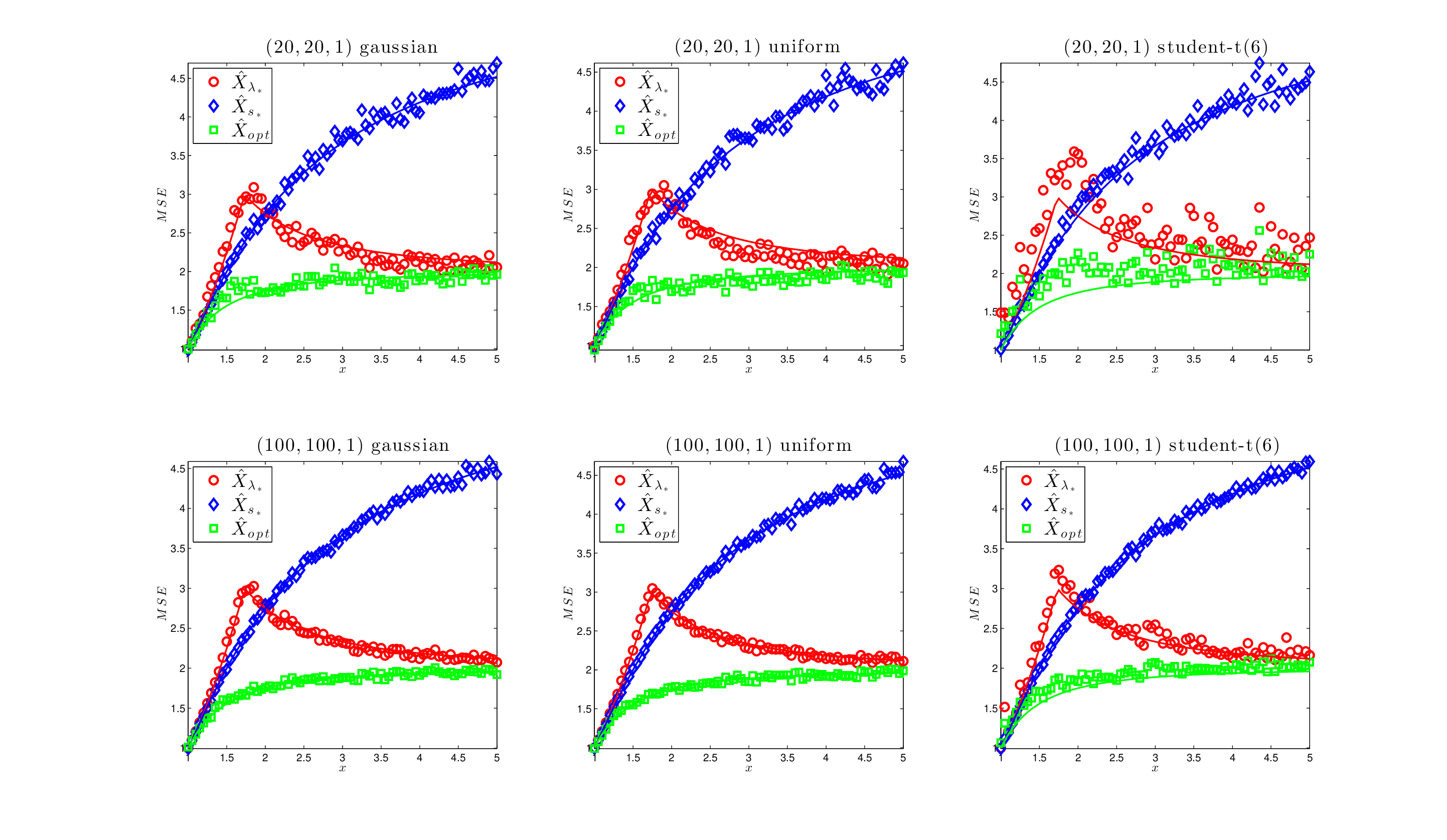}
  \caption{Comparison of asymptotic (solid line) 
    and empirically observed (dots) Frobenius loss for
  $n=20,100$ and $r=1$. Horizontal axis is the singular value of the signal
  matrix $X$. Shown are optimally tuned soft threshold $\hat{X}_{\lambda_*}$,
  optimally tuned hard threshold $\hat{X}_{s_*}$ and optimal shrinker
  $\hat{X}_{opt}$ from  \eqref{opt_shrinker_fro:eq}.
   This figure can be reproduced
 using the code supplement \cite{Gavish2015}.}
  \label{fig:fig7_a}
\end{figure}
\begin{figure}[h]
  \centering
  \includegraphics[width=6.5in,trim= 140 1 3 .5, clip=true]{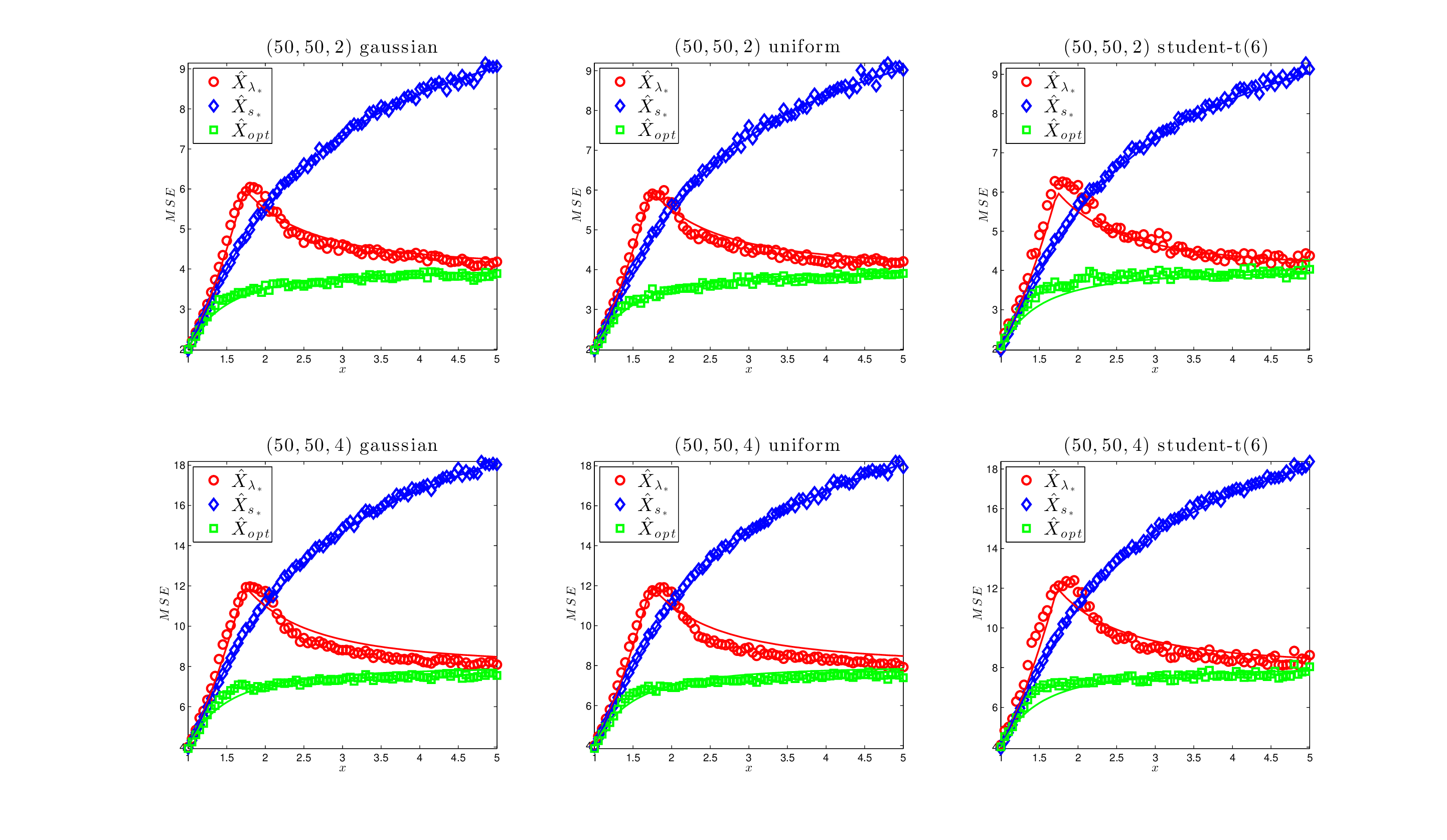}
  \caption{Comparison of asymptotic (solid line) 
    and empirically observed (dots) Frobenius loss for
  $n=50$ and $r=2,4$. 
Horizontal axis is the singular value of the signal
  matrix $X$. Shown are optimally tuned soft threshold $\hat{X}_{\lambda_*}$,
  optimally tuned hard threshold $\hat{X}_{s_*}$ and optimal shrinker
  $\hat{X}_{opt}$ from  \eqref{opt_shrinker_fro:eq}. This figure can be reproduced
 using the code supplement \cite{Gavish2015}.}
  \label{fig:fig7_b}
\end{figure}

\begin{figure}[h!]
\centerline{
  \includegraphics[width=7.5in]{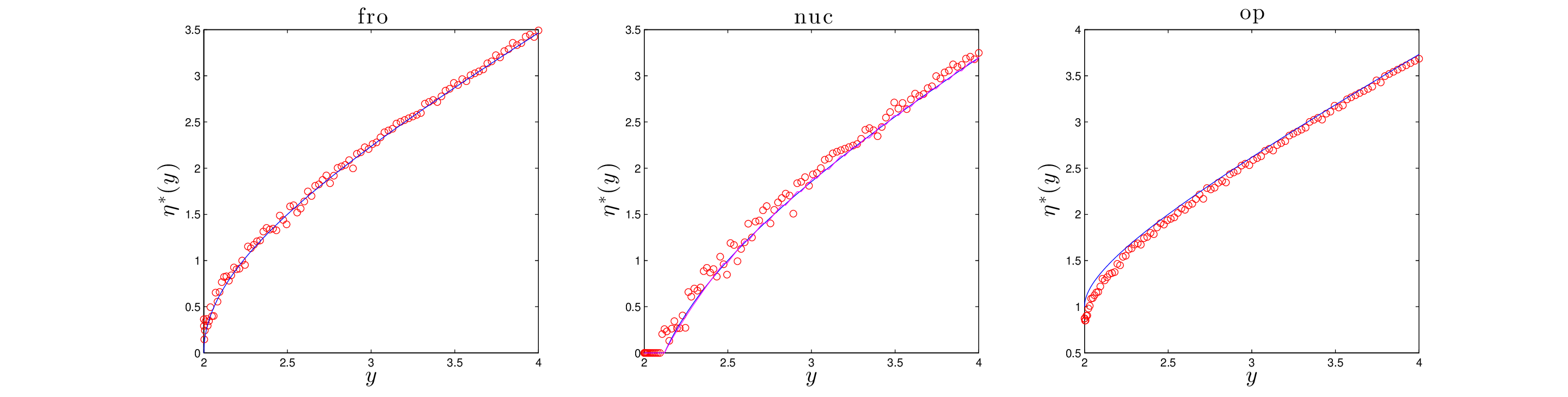}}
   \caption{\small 
     Optimal shrinkage computed by brute-force against the asymptotically
     optimal shrinkers from Section \ref{froopnuc:sec}, for $\beta=1$. This figure can be reproduced
 using the code supplement \cite{Gavish2015}.}
   \label{empirical_opt:fig}
 \end{figure}

\section{Conclusion}

We have presented a general framework for finding optimal shrinkers, either
analytically or numerically,  for a
variety of loss functions. 

Note that our general method, summarized in Theorem \ref{char:thm}, 
 is guaranteed to find a shrinker that is
 asymptotically unique admissible, or optimal, among {\em conservative} shrinkers (in
 the sense of Definition \ref{conservative:def}).
 This is an artifact of our proof method, and it is best to think of 
 Theorem \ref{char:thm} as a formal machine for finding ``good'' shrinkers, rather than a
 definite summary of their optimality properties. In fact, for all
 three loss functions considered in this paper, the optimal shrinkers we
 found dominate, in asymptotic loss, a much wider class of shrinkers. In
 particular, for all three losses, these optimal shrinkers 
 dominate the class of continuous shrinkers with the 
 property that $\eta(y)=0$ for all $y\leq \beta_+$, namely, shrinkers that
 truncate data singular values below the bulk edge $\beta_+$. In some sense,
 this is the class of ``reasonable'' shrinkers.
 
 The challenging issue is how to control the manner in which
 ``null'' singular values $y_{n,i}$ ($i>r$) affect the loss function. 
 When the noise distribution is Gaussian, 
 is possible to prove an analogy of Lemma \ref{residual:lem}, showing that the
 cumulative effect of these ``null'' singular values is negligible.
 To formally appeal to this fact, we are required to consider only loss functions
that enjoy a Lipschitz regularity property (on top of being decomposable and
 orthogonally invariant). 
  Then one can show that the optimal shrinkers characterized in 
  Theorem \ref{char:thm} dominate all ``reasonable'' shrinkers  as above. See
  \cite{Donohoa} for more details. 

Finally, we remark that closed-form solutions for the optimal shrinkers for Schatten-$p$ losses, 
and a generalization of our method to include Ky-Fan norms, both remain
interesting problems for further study.

\section*{Reproducible Research}

In the code supplement \cite{Gavish2015} we offer a Matlab software library that
includes:
\begin{enumerate}
  \item A function that calculates the optimal singular value shrinkage 
  w.r.t  
    the Frobenius, operator and nuclear norm losses, both in known or unknown noise
    level.
  \item Scripts that generate each of the figures in this paper.
  \item Notably, the script which generates Figure \ref{schatten1:fig} and
    Figure \ref{schatten2:fig} 
    includes an example of
    numerical evaluation of optimal shrinkers.
\end{enumerate}

\section*{Acknowledgements}

We thank Iain
Johnstone for helpful comments. We also thank Amit
Singer and Boaz Nadler for discussions stimulating this work, and Santiago
Velasco-Forero for pointing out an error in an earlier version of the
manuscript. We thank the anonymous referees for their helpful suggestions. 
This work was partially supported by NSF DMS-0906812 (ARRA).
MG was partially supported by a William R. and Sara Hart Kimball Stanford Graduate Fellowship.

\end{document}